\documentclass[11pt]{amsart}
\usepackage{amssymb,enumerate}
\usepackage{amsmath}
\usepackage{bbm}          
\usepackage{bm}
\usepackage{eso-pic,graphicx}
\usepackage{tikz}
\usepackage{esint}
\usepackage{ulem} 
\usepackage[colorlinks=true, pdfstartview=FitV, linkcolor=blue, citecolor=blue, urlcolor=blue]{hyperref} 
\usepackage{mathrsfs}

\newtheorem{theorem}{Theorem}[section]

\newtheorem{definition}[theorem]{Definition}

\newtheorem{lemma}[theorem]{Lemma}

\newtheorem{proposition}[theorem]{Proposition}
\newtheorem{remark}[theorem]{Remark}

\def\existingtheorem#1#2{
\noindent{\bf Theorem \ref{#1}.\hspace{4pt}}#2
}

\def\R{{\mathbb R}}

\def\C{{\mathbb C}}
\def\rn{{\mathbb R^n}}

\def\({\left(}
\def\){\right)}
\def\[{\left[}
\def\]{\right]}
\def\<{\left<}
\def\>{\right>}

\def\R{\mathbb R}
\def\C{\mathbb C}
\def\Z{\mathbb Z}
\def\N{\mathbb N}
\def\S{\mathscr S}

\def\less{\lesssim}

\DeclareSymbolFont{bbold}{U}{bbold}{m}{n}
\DeclareSymbolFontAlphabet{\mathbbn}{bbold}
\allowdisplaybreaks

\begin{document}

\title{John-Nirenberg Inequalities and Weight Invariant BMO Spaces}
\author{Jarod Hart}
\address{Higuchi Biosciences Center\\ University of Kansas\\ Lawrence, KS 66045}
\email{jvhart@ku.edu}

\author{Rodolfo H. Torres}
\address{Department of Mathematics\\ University of Kansas\\ Lawrence, KS 66045}
\email{torres@ku.edu}

\begin{abstract}
This work explores new deep connections between John-Nirenberg type inequalities and Muckenhoupt weight invariance for a large class of $BMO$-type spaces.  The results are formulated in a very general framework in which $BMO$ spaces are constructed using a base of sets, used also to define weights with respect to a non-negative measure (not necessarily doubling), and an appropriate oscillation functional. This  includes as particular cases many different function spaces on geometric settings of interest. As a consequence the weight invariance of several $BMO$ and Triebel-Lizorkin spaces considered in the literature is proved. Most of the invariance results obtained under this unifying approach are new even in the most classical settings.
\end{abstract}
\maketitle

\section{Introduction}

Spaces of Bounded Mean Oscillation ($BMO$) have been, and continue to be, of great interest and a subject of intense research in harmonic analysis.  One of the most fascinating aspects of $BMO$ spaces is their self-improvement properties, which go back to the work of John and Nirenberg in \cite{JN}.  

The space $BMO$ can be defined to be the collection of locally integrable functions $f$ such that
\begin{align*}
\|f\|_{BMO}=\sup_{Q\subset\R^n}\frac{1}{|Q|}\int_Q|f(x)-f_Q|dx<\infty.
\end{align*}
 Here, as usual,  $f_Q=\frac{1}{|Q|}\int_Qf(x)dx$ denotes the average of $f$ over the cube $Q$, and the supremum is taken over the collection $\mathcal Q$ of all cubes in $\R^n$ with sides parallel to the axes. The crucial property of $BMO$ functions is the John-Nirenberg inequality
 \begin{align*}
|\{x\in Q:|f(x)-f_Q|>\lambda\}|\leq c_1|Q| e^{-\frac{c_2\lambda}{\|f\|_{BMO}}}
\end{align*}
where $c_1$ and $c_2$ depend only on the dimension.

 A well-known immediate consequence of the John-Nirenberg inequality is the {\it $p$-power integrability}, which we also refer to as $p$-invariance,
$$
\|f\|_{BMO} \approx \sup_{Q\in\mathcal Q} \( \frac{1}{|Q|}\int_Q|f(x)-f_Q|^pdx\)^{1/p},
$$
for all $1< p<\infty$. 
Moreover, it can also be proved that the above  equivalence also hold for $0<p<1$ even though the right hand-side is not a norm in such a case. See for example the work of Str\"omberg \cite{stro} (or Lemma \ref{BMOp=BMO} below). 

The John-Nirenberg inequality also implies that $e^{|f(x)|/\rho}$ is  locally integrable for any $f\in BMO$ and an appropriate constant $\rho>0$. This exponential integrability led to a deep connection to Muckenhoupt weight theory, which in a rough sense says that $\log(A_2)=BMO$; here $A_p$ denotes the class of Muckenhoupt weights of index $1\leq p\leq\infty$.  More precisely, for every weight $w\in A_2$, $\log(w)\in BMO$, and for every $f\in BMO$ (real-valued), there exists $\lambda>0$ such that $e^{f/\lambda}\in A_2$; see for example the book of Garc\'{i}a-Cuerva and Rubio de Francia \cite{GCRdF}.  

Another deep connection was made between Muckenhoupt weights and $BMO$ in the work of Muckenhoupt and Wheeden \cite{MW}.  They proved that a function $f$ is in $BMO$ if and only if $f$ it is of bounded mean oscillation with respect to $w$ for all $w\in A_\infty$.  That is,  if for each 
$w\in A_\infty$, define $BMO_w$ to be the collection of all $w$-locally integrable functions $f$ such that
\begin{align*}
\|f\|_{BMO_w}=\sup_{Q\subset\R^n}\frac{1}{w(Q)}\int_Q|f(x)-f_{w,Q}|w(x)dx<\infty.
\end{align*}
Then,  $BMO=BMO_w$ and 
\begin{align}\label{added}
\|f\|_{BMO_w} \approx \|f\|_{BMO}.
\end{align}
Here $w(Q)=\int_Qw(x)dx$ is the $w$-measure of $Q$, and 
$$f_{w,Q}=\frac{1}{w(Q)} \int_Qf(x)w(x) dx.$$

Quantitative refinements of the above {\it weight invariant} result, were recently obtained by Hyt\"onen and P\'erez \cite{HP} and  Tsutsui \cite{Ts}, who gave precise control of the constants appearing  in \eqref{added}. 

\smallskip

Other weight invariant results in the literature include, for example, the work of Bui and Taibleson \cite{BT}, Harboure, Salinas, and Viviani \cite{HSV} and an article by the first author 
and Oliveira \cite{HO}.  In \cite[Theorem 3]{BT}, the authors show that the weighted endpoint Triebel-Lizorkin spaces $\dot F_{\infty,w}^{\alpha,\infty}$ coincide for all $w\in A_\infty$. In \cite[Proposition 4]{HSV}, the authors show that $BMO$, which agrees with the Triebel-Lizorkin space $\dot F_{\infty}^{0,2}$, continuously embeds into a weighted  version $\dot F_{\infty,w}^{0,2}$ for all $w\in A_2$ and $\dot F_{\infty}^{0,2}=\dot F_{\infty,w}^{0,2}$ for all 
$w\in A_1$ with comparable norms.  Moreover,  it is shown in \cite[Theorem 2.11]{HO} that for any $s>0$, 
$$\dot F_{\infty}^{s,2}=I_s(BMO)=I_s(BMO_w)=\dot F_{\infty,w}^{s,2},$$
 for all $w\in A_\infty$ with comparable norms; here $I_s(BMO)$ denotes the Sobolev-$BMO$ spaces defined by Neri \cite{N} and studied in depth by Strichartz \cite{Str1,Str2}.  

The main purpose of this article is to further explore the connections between John-Nirenberg type inequalities and Muckenhoupt weight invariance for generalized $BMO$ spaces and apply them to obtain several new results in a diversity of contexts.  These connections are made for $BMO$-type spaces formulated  in a very general setting involving a non-negative measure $\mu$, an oscillation functional $\Lambda$, a base of sets $\mathcal B$ 
used to define weights, some class of  ``functions" $\mathfrak F$, and a space $X(\mathcal B, \mu, \mathfrak F, \Lambda)=X(\mathcal B, \Lambda)= X$ consisting of  all the $f \in  \mathfrak F$  for which
\begin{align*}
\|f\|_X=\sup_{B\in\mathcal B}\frac{1}{\mu(B)}\int_B\Lambda(f,B)(x)d\mu(x)<\infty.
\end{align*}
For each $f\in\mathfrak F$ and $B\in\mathcal B$,  $\Lambda(f,B)(x)$ is assumed to be a non-negative, $\mu$-locally integrable function.  In principle, these objects generalize the role of $d\mu=dx$ the Lebesgue measure, $ \mathcal B=\mathcal Q$, and $\Lambda(f,Q)(x)=|f(x)-f_Q|$ for $f\in \mathfrak F= L^1_{loc}(\R^n)$, in which case of course $X(\mathcal Q,\Lambda)=BMO$.  However, there is much more flexibility in this definition.  For instance, we can apply our results in settings where $\mathcal B$ is the collection of balls, dyadic cubes, rectangles, dyadic rectangles, or other collections of sets, hence including various $BMO$ spaces in the literature associated to different geometries.  The results are also applicable in the situation where $\mathfrak F$ is the collection of locally integrable function, tempered distributions, or distributions modulo polynomials, and the functional $\Lambda$ takes various appropriate forms leading to distribution spaces like Triebel-Lizorkin spaces.  Moreover, the general approach we follow also works in a general $\sigma$-finite measure space and hence we obtain, not surprisingly, versions of results in the context of spaces of homogenous type, but more interestingly also in non-doubling settings where many tools of harmonic analysis fail. 

The main properties for $X(\mathcal B,\Lambda)$ of interest to us are weight invariance and John-Nirenberg $p$-power integrability or  $p$-invariance.  Our abstract setup is reminiscent of the one in the work by Franchi, P\'erez, and Wheeden in \cite{FPW} and others. The authors in \cite{FPW} considered inequalities of the form 
\begin{equation}\label{selfimprove}
\frac{1}{\mu(B)}\int_B |f(x)-f_B| d\mu(x)\lesssim \Gamma(f,B)
\end{equation}
for appropriate functionals $\Gamma$ taking values now in $\R$, and which lead to {\it self-improvements}.  In particular they investigated conditions on $\Gamma$ that allow the left-hand side of \eqref{selfimprove} to be replaced by its $p$-power version for $1<p<\infty$.  Several other authors (we shall give some references later on) have followed such abstract approaches too.  Nonetheless, and although one of our applications will overlap with results in \cite{FPW}, our results are geared more towards the establishment of several weighted norm equivalences that,  incidentally, will also work for $0<p<1$. We show that under remarkably minimal assumptions on the objects defining the spaces $X(\mathcal B,\Lambda)$, which apply to many situations, weight invariance and $p$-invariance are essentially equivalent properties. The precise statements are given in Theorems  \ref{smallnecessity}, \ref{necessity}, and 
\ref{sufficiency}. We then use known $p$-invariance results for several function spaces to prove their weight invariance as well. The following are a few examples of new invariance results obtained as corollaries of  Theorems \ref{smallnecessity} and \ref{necessity}, which showcase the convenience and benefits of the very general framework used. Here we state abridged versions of these results; see the corresponding theorems in Section 5 for their complete statements.

\medskip

\existingtheorem{wbmo}{{\it(The John-Nirenberg BMO space)} {\it Let $0<p<\infty$ and $w,v$ be in $ A_\infty$.  Then
\begin{align*}
\|f\|_{BMO}\approx\sup_{Q\subset\R^n}\(\frac{1}{w(Q)}\int_Q|f(x)-f_{v,Q}|^pw(x)dx\)^\frac{1}{p}.
\end{align*}}
}
\medskip

\existingtheorem{ndbmo}{{\it(BMO spaces with respect to non-doubling measures)}} {\it Let $\mu$ be a non-negative measure in $\rn$ such that $\mu(L)=0$ for any hyperplane $L$ orthogonal to one of the coordinate axes. Also let $0<p<\infty$ and $v,w\in A_\infty(\mu)$.  Then
\begin{align*}
\|f\|_{BMO_\mu(\rn)}&\approx\sup_{Q\in\mathcal Q}\(\frac{1}{\mu_w(Q)}\int_Q|f(x)-f_{\mu_v,Q}|^pd\mu_w(x)\)^{1/p}.
\end{align*}
}

\medskip

\existingtheorem{hbmo}{{\it(Duals of weighted Hardy spaces)}} {\it Let $0< r<\infty$, $w$ be in $A_1$, $v$ be in $A_\infty(w)$, and $\rho=v\cdot w$. Then for all $f$ in $BMO_{*,w}$
\begin{align*}
\|f\|_{{BMO}_{*,w}} :&=  \sup_{Q\in\mathcal Q} \frac{1}{w(Q)}\int_Q|f(x)-f_Q| dx\\
& \approx\sup_{Q\in\mathcal Q}\(\frac{1}{\rho(Q)}\int_Q|f(x)-f_Q|^rw(x)^{1-r}v(x)dx\)^{1/r}.
\end{align*}
}

\medskip

\existingtheorem{littlebmo}{{\it(Little $bmo$)}}{\it Let $0<p<\infty$ and $v$ and $w$ be rectangular weights in the class $ A_\infty(\R^{n_1}\times\R^{n_2})$.  Then
\begin{align*}
\|f\|_{bmo}  \approx\sup_{R\in\mathcal R}\(\frac{1}{w(R)}\int_R|f(x)-f_{v,R}|^pw(x)dx\)^{1/p}.
\end{align*}
}

\medskip  

\existingtheorem{tlbmo}{{\it(End-point Triebel-Lizorkin spaces)}}{\it Let $\alpha\in\R$, $0<p,q<\infty$, and $w$ in $A_\infty$.  Then for all $f$ in $\dot F_\infty^{\alpha,q}$
\begin{align*}
\|f\|_{\dot F_{\infty}^{\alpha,q}} &  \approx\sup_{Q\in\mathcal Q}\(\frac{1}{w(Q)}\int_Q\(\sum_{k\in\Z:2^{-k}\leq\ell(Q)}\!\!\!(2^{\alpha k}|\varphi_k*f(x)|)^q\)^\frac{p}{q}\!\!w(x)dx\)^\frac{1}{p}.
\end{align*}
}

\medskip  

The rest of this article is organized as follows. We introduce the terminology employed above and most of the notation used throughout the article in Section~2.   In Sections 3 and 4, we prove several results related to the equivalence between  the $p$-invariance and weight invariance for any $BMO$ type space $X(\mathcal B,\Lambda)$; in particular Theorems \ref{necessity} and \ref{sufficiency} alluded to before.  Finally in Section 5 we present the proofs of the theorems stated above, as well as some extensions of them and other applications.

\section{Definitions and Preliminaries}

 Let  $(Y, \Sigma, \mu)$ be a measure space, with $\mu$ non-negative and $\sigma$-finite, 
and let $\mathcal B$ be a collection of sets in the $\sigma$-algebra $\Sigma$ of $\mu$-measurable  sets such that  $0<\mu(B)<\infty$ for all $B\in \mathcal B$. Let  also $\mathcal M_\mathcal B$ be a sublinear operator that maps $L^1_{loc}(\mu)$ into non-negative $\mu$-measurable functions.  Here $L^1_{loc}(\mu)$ denotes the collection of all $\mu$-measurable functions $f:Y\rightarrow\C$ such that $\int_B|f(x)|d\mu(x)<\infty$ for all $B\in\mathcal B$.

\begin{definition}
Given $1<p<\infty$, a non-negative locally $\mu$-integrable function $w$ is in the Muckenhoupt class of weights $A_p(\mathcal B,\mu)=A_p(\mathcal B)$, 
or simply $A_p$, if 
\begin{align*}
[w]_{A_p(\mathcal B)}:=\sup_{B\in\mathcal B}\(\frac{1}{\mu(B)}\int_Bw(x)d\mu(x)\)\(\frac{1}{\mu(B)}\int_Bw(x)^{-p'/p}d\mu(x)\)^{p/p'} \!\!< \infty,
\end{align*}
where, as usual, $p'$ denotes the H\"older conjugate of $1<p<\infty$ given by $p'=\frac{p}{p-1}$.  

Define $A_\infty(\mathcal B)$ to be the union of all $A_p(\mathcal B)$ for $1<p<\infty$.  We say that $w\in A_1(\mathcal B)$ if $\mathcal M_{\mathcal B}w(x)\leq [w]_{A_1(\mathcal B)}w(x)$, where $[w]_{A_1(\mathcal B)}$ denotes the minimal constant satisfying this inequality.  

Define also for $1<\delta<\infty$ the reverse H\"older class $RH_\delta(\mathcal B)$  to be the collection of all $w\in A_\infty(\mathcal B)$ such that
\begin{align*}
\(\frac{1}{\mu(B)}\int_Bw(x)^\delta d\mu(x)\)^{1/\delta}\leq[w]_{RH_\delta}\frac{1}{\mu(B)}\int_Bw(x)d\mu(x),
\end{align*}
for all $B\in\mathcal B$.  Here again the reverse H\"older constant $[w]_{RH_\delta}$ for $w$ is the smallest constant such that the above inequality holds.  It will be necessary sometimes to specify the reverse H\"older class with respect to a measure $\mu$ and we will write 
$RH_\delta(\mathcal B, \mu)$  or  $RH_\delta( \mu)$ in such a case.
\end{definition}

\begin{remark}
{\rm In applications $\mathcal M_\mathcal B$ will be a  maximal function associated to the base $\mathcal B$. It is interesting, however, that we do not need $\mathcal M_\mathcal B$ to be so to prove our general results. Moreover even if $\mathcal M_\mathcal B$ is such a maximal function,  we are not assuming that the base is a Muckenhoupt base, i.e. that the $A_p$ classes characterize, or even suffice for, the boundedness of $\mathcal M_\mathcal B$ on weighted $L^p$ spaces.  On the other hand, we note that in the  above general setting, we will not be able to use any self-improving property on the weights unless we impose a reverse H\"older condition.}
\end{remark}

\begin{definition}\label{d:main}
Let $\mathfrak F$ be a set (typically of functions or distributions), and let $\Lambda$ be a mapping from $\mathfrak F\times\mathcal B$ into the collection of  $\mu$-measurable functions on $Y$, such that  $\Lambda(f,B)$ is non-negative for all $f\in\mathfrak F$ and $B\in\mathcal B$.  We consider the following spaces and properties.
\begin{enumerate}[(i)]
\item For $w\in A_\infty(\mathcal B)$ and $0<p<\infty$, $X_w^p(\mathcal B, \mu,\mathfrak F, \Lambda)$ is the collection of all $f\in\mathfrak F$ such that
\begin{align*}
\|f\|_{X_w^p}^p:=\sup_{B\in\mathcal B}\frac{1}{w(B)}\int_B\Lambda(f,B)(x)^pw(x)d\mu(x)<\infty,
\end{align*}
where $w(B)=\int_Bw(x)d\mu(x)$.   We will simply write $X_w^p(\mathcal B,\mu, \Lambda)$, more typically $X_w^p(\mathcal B,\Lambda)$, or just $X_w^p$, when the other objects in the definition of $X_w^p(\mathcal B, \mu,\mathfrak F, \Lambda)$ are clear from or do not play a significant role in the particular context being considered.
We will  also use the notation $X^p(\mathcal B,\Lambda)
=X_1^p(\mathcal B,\Lambda)$ when $w$ is the constant function $1$, $X_w(\mathcal B,\Lambda)=X_w^1(\mathcal B,\Lambda)$ when $p=1$, and $X(\mathcal B,\Lambda)=X_1^1(\mathcal B,\Lambda)$ when both $w=1$ and $p=1$. We will call $X_w^p(\mathcal B,\Lambda)$ an ``oscillation space" and $\Lambda$ an ``oscillation functional".

\medskip

\item The oscillation space $X(\mathcal B,\Lambda)$ satisfies the weight invariance property for a collection of weights $\mathcal W\subset A_\infty(\mathcal B)$ if $X_w(\mathcal B,\Lambda)=X_v(\mathcal B,\Lambda)$ for all $w,v\in \mathcal W$, and there is a constant $b_{w,v}>0$ such that $$\|f\|_{X_v}\leq b_{w,v}\|f\|_{X_w}$$
for all $f\in X_w(\mathcal B,\Lambda)$.  Without loss of generality, we let $b_{w,v}$ be the smallest such constant.

\medskip

\item For an oscillation space $X(\mathcal B,\Lambda)$ that satisfies the weight invariance property for all weights in $A_p(\mathcal B)$, define for 
$t\geq 1$ the function
\begin{align}
\Psi_p(t):=\sup\{b_{1,v}:v\in A_p(\mathcal B),\;\;[v]_{A_p}\leq t\}.\label{Psi}
\end{align}

\medskip

\item The oscillation space $X_w(\mathcal B,\Lambda)$ satisfies the $p$-power John-Nirenberg property, or $p$-invariance property, for an interval $I\subset(0,\infty)$ if for any $p,q\in I$ with $p<q$ there exists a constant  $C(w) >0$ (which may also depend on $p$ and $q$)  
such that 
 $$\|f\|_{X_w^q}\leq C(w) \|f\|_{X_w^p}$$
for all $f\in X_w(\mathcal B,\Lambda)$.  We let $c_{p,q}(w)$ be the infimum of the constants such that the above inequality holds.  When $w=1$, we use the notation $c_{p,q}=c_{p,q}(1)$.  
\medskip
\end{enumerate}
\end{definition}

Throughout this article, we reserve the lower case subscripted letters $b_{w,v}$ and $c_{p,q}(w)$, to play the role they do in the definitions above.  

\medskip

\begin{remark}\label{r:p<1}
{\rm
Note that since $\Lambda(f,B)$ is a non-negative $\mu$-measurable function, it follows that $\Lambda(f,B)^p\, w$ is also a non-negative $\mu$-measurable function for all $f\in\mathfrak F$, $B\in\mathcal B$, $0<p<\infty$, and $w\in A_\infty(\mathcal B,\mu)$.  Hence the integral used to compute $\|\cdot\|_{X_w^p}$ is well-defined, though possibly infinite for some elements $f\in\mathfrak F$. 

Note also that despite the notation $\|\cdot\|_{X_w^p}$ and the name oscillating space,  the collection $X_w^p(\mathcal B,\Lambda)$ may not be a Banach space, a normed space, or even a vector space.  We use this notation because it is conducive to think of it in this context, even though we do not need linearity, completeness, or a normed space structure for our computations.   If $X$ and $Z$ are either oscillation spaces or normed function spaces, the notation $X \approx Z$ will always mean that $X=Z$ as sets and $\|f\|_X \approx \|f\|_Z$ for all of their elements.

}
\end{remark}

\begin{remark}
{\rm   The reason for introducing the technically looking function $\Psi_p$ in \eqref{Psi} will become apparent in Section~4, where it will be used to quantify some uniform estimates.  Note that for each $p$, $\Psi_p(t)$ is obviously a non-decreasing function of $t$  that in principle could take the value $\infty$ for some or all $t\geq 1$.}
\end{remark}

\begin{remark}\label{pinvremark}
{\rm 
It is important to remark that the inequality
$$\|f\|_{X_w^q}\leq c_{p,q}(w)\|f\|_{X_w^p}$$
in {\it(iv)} is only required for $f\in X_w(\mathcal B,\Lambda)$, not for $f\in X_w^p(\mathcal B,\Lambda)$ or $f\in X_w^q(\mathcal B,\Lambda)$.  Hence, there is a distinction between the $p$-power John-Nirenberg property when $p\geq 1$ and when $p<1$.  
 
Indeed, suppose $X(\mathcal B,\Lambda)$ satisfies the p-power John-Nirenberg property for $[1,\infty)$.  This means that $\|f\|_{X^p}\leq c_{1,p}\|f\|_{X}$ for all $1<p<\infty$ and $f\in X(\mathcal B,\Lambda)$; hence $X(\mathcal B,\Lambda)\subset X^p(\mathcal B,\Lambda)$.  It is immediate by H\"older's (or Young's) inequality that $\|f\|_{X}\leq\|f\|_{X^p}$, and hence $X^p(\mathcal B,\Lambda)\subset X(\mathcal B,\Lambda)$.  Therefore the John-Nirenberg property for $X(\mathcal B,\Lambda)$ on $[1,\infty)$ implies that $X(\mathcal B,\Lambda)=X^p(\mathcal B,\Lambda)$ and $\|\cdot\|_{X}\approx\|\cdot\|_{X^p}$ for all $1\leq p<\infty$.

On the other hand, when $X(\mathcal B,\Lambda)$ satisfies the $p$-power John-Nirenberg property for $(0,1]$, we cannot make such a strong conclusion.  In this case, it follows that $\|f\|_{X^p}\leq \|f\|_X\leq c_{p,1}\|f\|_{X^p}$ for all $f\in X(\mathcal B,\Lambda)$.  The restriction to ``functions'' in $X(\mathcal B,\Lambda)$ is of substance here.  We cannot conclude that $f\in X^p(\mathcal B,\Lambda)$ implies $f\in X(\mathcal B,\Lambda)$ or that $X(\mathcal B,\Lambda)=X^p(\mathcal B,\Lambda)$ in this case. We can only conclude that $\|\cdot\|_X\approx\|\cdot\|_{X^p}$  for all $0<p<1$ when restricted to $X(\mathcal B,\Lambda)$. }
\end{remark}

\begin{remark}
  {\rm 
  There is a plethora of $BMO$ spaces that can be realized as  $X(\mathcal B,\mu,\mathfrak F,\Lambda)$ spaces with the appropriate choices of $\mathcal B$, $\mu$, $\mathfrak F$, and $\Lambda$.  For example, every space mentioned in the introduction can be realized as such a space in an obvious way suggested by their standard definitions.  More details are given  in Section~5.}
\end{remark}

We will now lay out several assumptions on $\mu$, $\mathcal B$, and $\mathcal M_{\mathcal B}$.  Different subsets of these assumptions will be used to prove different results. We will refer to them throughout as assumptions {\it A1}--{\it A4}, defined as follows.

\medskip

{\it
\begin{enumerate}
\renewcommand{\labelenumi}{\textbf{S.\theenumi}}
\item[A1.] $A_1(\mathcal B)\subset A_p(\mathcal B)$ and $[w]_{A_p}\leq [w]_{A_1}$ for all $1<p<\infty$.
\item[A2.] $A_\infty(\mathcal B)  =  \bigcup_{1<\delta<\infty}RH_\delta(\mathcal B)$. 
\item[A3.] For all $1<p<\infty$, the operator $\mathcal M_{\mathcal B}$ satisfies $0<\|\mathcal M_{\mathcal B}\|_{L^p,L^p}<\infty$ and 
$\limsup_{p\rightarrow1^+}\|\mathcal M_\mathcal B\|_{L^p,L^p} =\infty$. 
\item[A4.] 
There exist functions $\Delta :(1,\infty)^2 \rightarrow(1,\infty)$, non-increasing in each variable, and $ K:(1,\infty)^2\rightarrow(1,\infty)$, non-decreasing in each variable,  such that if $u\in A_p(\mathcal B)$ then $u\in RH_{\Delta (p, [u]_{A_p})}(\mathcal B)$ with 
$$[u]_{RH_{\Delta (p, [u]_{A_p})}}\leq K(p, [u]_{A_p})$$
 for all $1<p<\infty$. 
\end{enumerate}
}

\begin{remark}\label{Delta+K}
{\rm It is well-known that assumptions {\it A1}--{\it A4} hold in many situations when $\mathcal M_\mathcal B$ is the centered or uncentered Hardy-Littlewood maximal function associated to $\mathcal B$.  It is worth noting that in the non-doubling setting, it is necessary to choose $\mathcal M_\mathcal B$ to be the centered maximal operator to assure {\it A3} holds, since the uncentered Hardy-Littlewood maximal operator associated to a non-doubling measure $\mu$ may not be $L^p(\mu)$ bounded for $1<p<\infty$. Property {\it A2} essentially represents the existence of a reverse H\"older inequality, while {\it A4} is a quantified version of that.  Actually {\it A4} implies {\it A2}, but we list both since only {\it A2} will be needed in some of our arguments.   We direct the reader to \cite{HPR,HP,LPR} and the reference therein for more information on the recent interest in sharp quantitative versions of the reverse H\"older inequality.
 
 In all the cases in which we are aware {\it A4} holds, $\Delta$ and $K$ can be taken in the form $\Delta(p,t)=1+\frac{1}{\tau(p)t}$ and $K(p,t)=C$,  for an appropriate function $\tau(p)$ and constant $C>0$.  For example, the following choices for $\Delta$ and $K$ are sufficient to satisfy {\it A4} in the corresponding settings:
\begin{itemize}
\item $\Delta(p,t)=1+\frac{1}{2^{n+1}t-1}$ and $K(p,t)=2$ for the standard Euclidean setting using cubes and the Lebesgue measure (see \cite[Theorem 2.3]{HPR}).

\item $\Delta(p,t)=1+\frac{1}{2^{p+2}t}$ and $K(p,t)=2$ for the Euclidean setting using rectangles and the Lebesgue measure (see \cite[Theorem 1.2]{LPR}).

\item $\Delta(p,t)=1+\frac{1}{\tau t}$ and $K(p,t)=C$ for the space of homogeneous type setting, where $\tau$ and $C$ are fixed constants depending on parameters of the underlying space (see \cite[Theorem 1.1]{HPR}).

\item  $\Delta(p,t)=1+\frac{1}{2^{p+1}B(n)t}$ and $K(p,t)=2$ for the Euclidian setting using non-doubling measures, where $B(n)$ is the Besicovitch constant for $\R^n$ (see \cite[Lemma 2.3]{OP} and the remarks in the  introduction of \cite{LPR}).

\end{itemize}
Actually, in some of these examples we can take $K(p,t)=2^{1/\Delta(p,t)}$. However this dependance on $p$ and $t$ is of no relevance since 
$\Delta(p,t)\geq 1$ and hence $1\leq K(p,t) \leq 2$. Therefore, for simplicity, we shall use $K(p,t)= 2$.
}
\end{remark}

\section{Sufficient Conditions for Weight Invariance}

In this section we provide sufficient conditions for weight invariance of $X(\mathcal B,\Lambda)$.  The first theorem of this section holds solely based on the definition of the $A_p(\mathcal B)$ weights, while the second one only needs as an additional assumption {\it A2}.  
Proposition \ref{c:Psiestimate} needs the stronger assumption {\it A4}.

\begin{lemma}\label{l:smallpJN} 
Fix $w\in A_\infty(\mathcal B)$.  If there exist $0<q_0<p_0<\infty$, with $q_0\leq 1$ such that $X_w(\mathcal B,\Lambda)$ satisfies the $p$-power John-Nirenberg property on the interval $[q_0,p_0]$, then $X_w(\mathcal B,\Lambda)$ also satisfies the $p$-power John-Nirenberg property on the interval $(0,p_0]$.
\end{lemma}

\begin{proof}
Let  $$ q= \inf \{r>0: X_w(\mathcal B,\Lambda) \text {\, is $p$-invariant on\,}  [r,p_0]\}.$$  Clearly $0\leq q\leq q_0$. By way of contradiction, assume that $q>0$, and let $\epsilon>0$  and $r>q$ be selected so that  $q<r-\epsilon<p_0$ but $0<r-2\epsilon<q$. Then, for any $f \in X_w(\mathcal B,\Lambda)$,
\begin{align*}
\|f\|_{X_w^{r-\epsilon}}^{r-\epsilon}&\leq \sup_B\frac{1}{w(B)}\|\Lambda(f,B)^{r/2-\epsilon}\|_{L_w^2(B)}\|\Lambda(f,B)^{r/2}\|_{L_w^2(B)}\\
&\leq\|f\|_{X_w^{r-2\epsilon}}^{(r-2\epsilon)/2}\|f\|_{X_w^r}^{r/2}\\
&\leq c_{r-\epsilon,r}^{r/2}(w) \|f\|_{X_w^{r-2\epsilon}}^{(r-2\epsilon)/2}\|f\|_{X_w^{r-\epsilon}}^{r/2}.
\end{align*}
 Because of the invariance on $[q_0,p_0]$ and the fact that $q_0 \leq1$, we also have
$$\|f\|_{X_w^{r-\epsilon}} \lesssim \|f\|_{X_w^{q}} \leq \|f\|_{X_w} < \infty.$$
Therefore
\begin{align*}
\|f\|_{X_w^{r-\epsilon}}&\leq c_{r-\epsilon,r}^\frac{r}{r-2\epsilon} (w)  \|f\|_{X_w^{r-2\epsilon}},
\end{align*}
contradicting the definition of $q$.
\end{proof}

The next theorem shows that $p$-invariance in a limited range of exponents implies appropriate weight invariance.

\begin{theorem}\label{smallnecessity}

If $X(\mathcal B,\Lambda)$ satisfies the $p$-power John-Nirenberg property for $[1,p_0]$ for some $p_0>1$ and $X(\mathcal B,\Lambda)=X^p(\mathcal B,\Lambda)$ for all $0<p\leq 1$, then $X(\mathcal B,\Lambda)$ satisfies the weight invariant property for $RH_{p_0'}(\mathcal B)$.  In particular, for each  $1<q<\infty$ such that $w\in A_q(\mathcal B)\cap RH_{p_0'}(\mathcal B)$, it holds that  
\begin{align*}
&(c_{1/q,1}[w]_{A_q})^{-1}\|f\|_X\leq\|f\|_{X_w}\leq c_{1,p_0}[w]_{RH_{p_0'}}\|f\|_X
\end{align*}
for all $f\in X(\mathcal B,\Lambda)$.

If in addition $w\in  RH_\sigma$ for some $\sigma>p_0'$, then $X_w(\mathcal B,\Lambda)$ satisfies the $p$-power John-Nirenberg property for $(0,p_0/\sigma'\,]$ and 
\begin{align*}
&c_{1,p_0/\sigma'}(w)\leq c_{1,p_0}c_{1/q,1} \( [w]_{RH_\sigma}\)^{\frac{\sigma}{p_0(\sigma-p_0')}} [w]_{A_q}.
\end{align*}
\end{theorem}

\begin{proof}
Assume $X(\mathcal B,\Lambda)$ satisfies the $p$-power John-Nirenberg property for $[1,p_0]$ for some $p_0>1$.  Fix $w \in RH_{p_0'}(\mathcal B)$.   Then for $f\in X(\mathcal B,\Lambda)$
\begin{align*}
&\|f\|_{X_w}=\sup_{B\in\mathcal B}\frac{1}{w(B)}\int_B\Lambda(f,B)(x)w(x)d\mu(x)\\
& \leq\sup_{B\in\mathcal B}\frac{\mu(B)}{w(B)}\(\frac{1}{\mu(B)}\int_B\Lambda(f,B)(x)^{p_0}d\mu(x)\)^{\frac{1}{p_0}}  \(\frac{1}{\mu(B)}\int_Bw(x)^{p_0'} d\mu(x)\)^{\frac{1}{p_0'}}\\
&\leq [w]_{RH_{p_0'}} c_{1,p_0} \|f\|_{X}.
\end{align*}
Hence $X(\mathcal B,\Lambda) \subset X_w(\mathcal B,\Lambda)$.

Let $w\in A_q(\mathcal B)$ for some $1<q<\infty$.  For $f\in X_w(\mathcal B,\Lambda)$, it follows that
\begin{align*}
\|f\|_{X^{1/q}}&=\sup_{B\in\mathcal B}\(\frac{1}{\mu(B)}\int_B\Lambda(f,B)(x)^{1/q}w(x)^{1/q}w(x)^{-1/q}d\mu(x)\)^q\\
&\leq\sup_{B\in\mathcal B}\(\frac{1}{w(B)}\int_B\Lambda(f,B)(x)w(x)d\mu(x)\)  \times \\
& \hspace{3cm} \(\frac{w(B)}{\mu(B)}\)\(\frac{1}{\mu(B)}\int_Bw(x)^{-q'/q}d\mu(x)\)^{q/q'}\\
& \leq [w]_{A_q}\|f\|_{X_w},
\end{align*}
and so $f\in X^{1/q}(\mathcal B,\Lambda)$. By assumption $X^{1/q}(\mathcal B,\Lambda)=X(\mathcal B,\Lambda)$, and 
therefore $X(\mathcal B,\Lambda)=X_w(\mathcal B,\Lambda)$ for all $w\in RH_{p_0'}(\mathcal B)$.  It also follows that 
\begin{align*}
\|f\|_X\leq c_{1/q,1}\|f\|_{X^{1/q}}\leq c_{1/q,1}[w]_{A_q}\|f\|_{X_w}.
\end{align*}

By modifying some of the arguments we can obtain the $p$-invariance of $X_w(\mathcal B,\Lambda)$.  Indeed, let $1< p<p_0$ and $w\in RH_{\frac{p_0}{p_0-p}}(\mathcal B)$.  Note that  $$\(\frac{p_0}{p}\)'=\frac{p_0}{p_0-p}> p_0'.$$  Hence $w\in RH_{p_0'}$ as well, and so the first part of this theorem is applicable here.   Then for $f\in X(\mathcal B,\Lambda)= X_w(\mathcal B,\Lambda)$, 
\begin{align*}
\|f\|_{X_w^p} & =\sup_{B\in\mathcal B}\(\frac{1}{w(B)}\int_B\Lambda(f,B)(x)^pw(x)d\mu(x)\)^{\frac{1}{p}}\\
& \leq\sup_{B\in\mathcal B}  \(\frac{\mu(B)}{w(B)} \)^{\frac{1}{p}}
\(\frac{1}{\mu(B)}\int_B\Lambda(f,B)(x)^{p_0}d\mu(x)\)^{\frac{1}{p_0}}  \times\\
& \hspace{55mm}  \(\frac{1}{\mu(B)}\int_Bw(x)^\frac{p_0}{p_0-p} d\mu(x)\)^{\frac{p_0-p}{p_0 p}}\\
& \leq  [w]_{RH_\frac{p_0}{p_0-p}}^{\frac{1}{p}} c_{1,p_0} \|f\|_{X}\\
&\leq c_{1/q,1}  c_{1,p_0} [w]_{RH_\frac{p_0}{p_0-p}}^{\frac{1}{p}} [w]_{A_q}\|f\|_{X_w}.
\end{align*}
Therefore $X_w(\mathcal B,\Lambda)$ satisfies the $p$-power John-Nirenberg property on $[1,p]$ when $w\in RH_\frac{p_0}{p_0-p}$, and by Lemma \ref{l:smallpJN} it does so on $(0,p]$ too.
\end{proof}

We now show that the full $p$-invariance for one weight, implies $p$-invariance and weight invariance for all weights in $A_\infty(\mathcal B)$.

\begin{theorem}\label{necessity}
Suppose $\mu$ and $\mathcal B$ satisfy the assumption {\it A2}.  
Assume that for some $w_0\in A_\infty(\mathcal B)$, $X_{w_0}(\mathcal B,\Lambda)$ satisfies the $p$-power John-Nirenberg property for the interval $[1,\infty)$, and that $X_{w_0}^p(\mathcal B,\Lambda)=X_{w_0}(\mathcal B,\Lambda)$ for all $0<p< 1$.  Then $X(\mathcal B,\Lambda)$ satisfies the weight invariance property for $A_\infty(\mathcal B)$, and $X_w(\mathcal B,\Lambda)$ satisfies the $p$-power John-Nirenberg property for $(0,\infty)$ for every $w\in A_\infty(\mathcal B)$.  Furthermore, if $w_0\in A_p(\mathcal B)\cap RH_\sigma(\mathcal B)$ and $w\in A_q(\mathcal B)\cap RH_\delta(\mathcal B)$ for $1<p,q,\delta,\sigma<\infty$, then 
\begin{align*}
&b_{w_0,w}\leq c_{1,p\delta'}(w_0)[w_0]_{A_p}^\frac{1}{p\delta'}[w]_{RH_\delta},\\
&b_{w,w_0}\leq c_{(q\sigma')^{-1},1}(w_0)[w]_{A_q}[w_0]_{RH_\sigma}^{q\sigma'},\text{ and}\\
&c_{1,t}(w)\leq c_{1,tp\delta'}(w_0)b_{w,w_0}[w_0]_{A_p}^\frac{1}{p\delta't}[w]_{RH_\delta}^\frac{1}{t}
\end{align*}
for all $1<t<\infty$.
\end{theorem}

\begin{proof}
Assume that $w_0\in A_\infty(\mathcal B)$, that $X_{w_0}(\mathcal B,\Lambda)$ satisfies the $p$-power John-Nirenberg inequality for the interval $[1,\infty)$, and that $X_{w_0}^p(\mathcal B,\Lambda)=X_{w_0}(\mathcal B,\Lambda)$ for all $0<p<1$.  In particular,  by assumption {\it A2},  $w_0\in A_{p}(\mathcal B) \cap RH_\sigma(\mathcal B)$ for some $1<p, \sigma<\infty$.  Also let $w\in A_\infty(\mathcal B)$ with $w\in A_q(\mathcal B) \cap RH_\delta(\mathcal B)$ for some $1<q, \delta<\infty$, where again such $\delta$ exists by assumption {\it A2}.  Define 
$$
r=p\delta'>1 \text{\, and \,} s=\delta'+\frac{p'\delta'}{p\delta}>1.
$$  
These numbers are chosen so that 
$$
sr'/r=p'/p \text{\, and \,}s'r'=\delta.  
$$
Then for any $f\in X_{w_0}(\mathcal B,\Lambda)$ we have
\begin{align*}
& \|f\|_{X_w}=\sup_{B\in\mathcal B}\frac{1}{w(B)}\int_B\Lambda(f,B)(x)w_0(x)^{1/r}w_0(x)^{-1/r}w(x)d\mu(x)\\
&\leq\sup_{B\in\mathcal B}\frac{1}{w(B)}\(\int_B\Lambda(f,B)(x)^rw_0(x)d\mu(x)\)^{\frac{1}{r}}\!\!
\(\int_Bw_0(x)^{-r'/r}w(x)^{r'}d\mu(x)\)^{\frac{1}{r'}}\\
&\leq\|f\|_{X_{w_0}^r}\sup_{B\in\mathcal B}\frac{w_0(B)^{1/r}}{w(B)}\(\int_Bw_0(x)^{-sr'/r}d\mu(x)\)^\frac{1}{sr'}\(\int_Bw(x)^{s'r'}d\mu(x)\)^{\frac{1}{s'r'}}\\
&\leq c_{1,r}(w_0)\|f\|_{X_{w_0}}\sup_{B\in\mathcal B}\frac{\mu(B)}{w(B)}\frac{w_0(B)^{1/r}}{\mu(B)^{1/r}}\\
&\hspace{2cm} \times \(\frac{1}{\mu(B)}\int_Bw_0(x)^{-p'/p}d\mu(x)\)^\frac{p}{p'r}\(\frac{1}{\mu(B)}\int_Bw(x)^{\delta}d\mu(x)\)^{\frac{1}{\delta}}\\
&\leq c_{1,r}(w_0)[w_0]_{A_p}^\frac{1}{r}[w]_{RH_\delta}\|f\|_{X_{w_0}}=c_{1,p\delta'}(w_0)[w_0]_{A_p}^\frac{1}{p\delta'}[w]_{RH_\delta}\|f\|_{X_{w_0}}
\end{align*}
Therefore $X_{w_0}(\mathcal B,\Lambda)\subset X_w(\mathcal B,\Lambda)$ and 
$$b_{w_0,w}\leq c_{1,p\delta'}(w_0)[w_0]_{A_p}^\frac{1}{p\delta'}[w]_{RH_\delta}.$$
 Define 
 $$
 r=q\sigma'>1\text{\, and \,}  s=\sigma'+\frac{q'\sigma'}{q\sigma}>1,
 $$
where now these numbers are chosen so that 
$$
sr'/r=q'/q \text{\, and \,}s'r'=\sigma. 
$$
It follows that, for $f\in X_w(\mathcal B,\Lambda)$, 
\begin{align*}
& \|f\|_{X_{w_0}^{1/r}}=\sup_{B\in\mathcal B}\(\frac{1}{w_0(B)}\int_B\Lambda(f,B)(x)^{1/r}w_0(x)w(x)^{1/r}w(x)^{-1/r}d\mu(x)\)^r\\
&\leq \sup_{B\in\mathcal B}\frac{1}{w_0(B)^r}\(\int_B\Lambda(f,B)(x)w(x)d\mu(x)\)\(\int_Bw_0(x)^{r'}w(x)^{-r'/r}d\mu(x)\)^{r/r'}\\
&\leq \|f\|_{X_w}\sup_{B\in\mathcal B}\frac{\mu(B)^{r-1}w(B)}{w_0(B)^r}\\
& \hspace{2cm} \times \(\frac{1}{\mu(B)}\int_Bw_0(x)^\sigma d\mu(x)\)^{r/\sigma}\(\frac{1}{\mu(B)}\int_Bw(x)^{-q'/q}d\mu(x)\)^{q/q'}\\
&\leq \|f\|_{X_w}\sup_{B\in\mathcal B}\frac{\mu(B)^{r-1}w(B)}{w_0(B)^r}\([w_0]_{RH_\sigma}\frac{w_0(B)}{\mu(B)}\)^r\([w]_{A_q}\frac{\mu(B)}{w(B)}\)\\
&\leq [w]_{A_q}[w_0]_{RH_\sigma}^r\|f\|_{X_w}=[w]_{A_q}[w_0]_{RH_\sigma}^{q\sigma'}\|f\|_{X_w}.
\end{align*}
Therefore $X_w(\mathcal B,\Lambda)\subset X^{1/r}_{w_0}(\mathcal B,\Lambda)$.  By assumption, $X_{w_0}(\mathcal B,\Lambda)=X^{1/r}_{w_0}(\mathcal B,\Lambda)$, and hence it follows that $X_{w_0}(\mathcal B,\Lambda)=X_w(\mathcal B,\Lambda)$.  It also follows that 
$$b_{w,w_0}\leq c_{(q\sigma')^{-1},1}[w]_{A_q}[w_0]_{RH_\sigma}^{q\sigma'}.$$

It remains to be shown that for any  $w\in A_\infty(\mathcal B)$,  $X_{w}(\mathcal B,\Lambda)$ is  also $p$-invariant. So let again 
 $w_0\in A_p(\mathcal B)\cap RH_\sigma(\mathcal B)$ and $w\in A_q(\mathcal B)\cap RH_\delta(\mathcal B)$ for $1<p,q,\delta,\sigma<\infty$, and
 $$r=p\delta' \text {\, and \,} s=1+p'(\delta'-1),$$
 so that as before
 $$sr'/r=p'/p \text {\, and \,} s'r'=\delta.$$
   Fix $1<t<\infty$ and $f\in X_{w}(\mathcal B,\Lambda)$. We already know that $X_w(\mathcal B,\Lambda)= X_{w_0}(\mathcal B,\Lambda)$ and that 
   $X_{w_0}(\mathcal B,\Lambda)$ is $p$-invariant for the interval $[1,\infty)$, hence  
\begin{align*}
& \frac{1}{w(B)}\int_B\Lambda(f,B)(x)^tw(x)d\mu(x)\\
&\leq\frac{1}{w(B)}\(\int_B\Lambda(f,B)(x)^{tr}w_0(x)d\mu(x)\)^{1/r}\(\int_Bw_0(x)^{-r'/r}w(x)^{r'} d\mu(x)\)^{1/r'}\\
&\leq c_{1,tr}(w_0)^t\|f\|_{X_{w_0}}^t\frac{w_0(B)^{1/r}}{w(B)} \\
& \hspace{2cm} \times \(\int_Bw_0(x)^{-sr'/r} d\mu(x)\)^{1/(sr')}\(\int_Bw(x)^{s'r'} d\mu(x)\)^{1/(s'r')}\\
&\leq c_{1,tr}(w_0)^tb_{w,w_0}^t\|f\|_{X_{w}}^t\frac{w_0(B)^{1/r}}{w(B)} \\
& \hspace{2.5cm} \times \(\int_Bw_0(x)^{-p'/p} d\mu(x)\)^{p/(rp')}\(\int_Bw(x)^{\delta} d\mu(x)\)^{1/\delta}\\
&\leq c_{1,tr}(w_0)^tb_{w,w_0}^t\|f\|_{X_{w}}^t[w_0]_{A_p}^\frac{1}{p\delta'}[w]_{RH_\delta}.
\end{align*}
Therefore $c_{1,t}(w)\leq c_{1,tp\delta'}(w_0)b_{w,w_0}[w_0]_{A_p}^\frac{1}{p\delta't}[w]_{RH_\delta}^\frac{1}{t}$.
\end{proof}

In the next result we impose more structure on the behavior of $\mu$ and $\mathcal B$ through the assumption {\it A4}.  In this way, we are able to ensure that $\Psi_p$, as defined in \eqref{Psi}, is finite when $X(\mathcal B,\Lambda)$ satisfies the $p$-power John-Nirenberg property. The proposition below essentially says that if we can control in a quantitative way the reverse H\"older constant for a weight by its $A_p$ norm, then we can also quantify the $b_{1,w}$ constants in the weight variance estimates.

\begin{proposition}\label{c:Psiestimate}
Suppose $\mu$ and $\mathcal B$ satisfy the assumption {\it A4}.  If 
$X(\mathcal B,\Lambda)$ satisfies the $p$-power John-Nirenberg property for $[1,\infty)$, then 
$$\Psi_p(t)\leq K(p,t)c_{1,\Delta (p,t)'} $$
 for all $1<p<\infty$.
\end{proposition}

\begin{proof}
Let $1<p<\infty$, $w\in A_p(\mathcal B)$ with $[w]_{A_p}\leq t$.  By  {\it A4}, we have that $w\in RH_{\Delta (p,[w]_{A_p})} $ and
$ [w]_{RH_{\Delta (p,[w]_{A_p})}} \leq   K(p, [w]_{A_p})$. Since $\Delta$ is non-increasing in each variable, it follows that $RH_{\Delta (p,[w]_{A_p})} \subset 
RH_{\Delta (p,t)}$ and $ [w]_{RH_{\Delta (p,t)}} \leq    [w]_{RH_{\Delta (p,[w]_{A_p})}}$.  
However, since $K$ is non-decreasing in each variable $K(p, [w]_{A_p})  \leq  K(p, t)$, and hence 
$ [w]_{RH_{\Delta (p,t)}} \leq   K(p,t)$. 
Therefore,
\begin{align*}
& \frac{1}{w(B)}\int_B\Lambda(f,B)(x)w(x)d\mu(x)\\
&\leq\frac{\mu(B)}{w(B)}\(\frac{1}{\mu(B)}\int_B\Lambda(f,B)(x)^{\Delta (p,t)'}d\mu\)^\frac{1}{\Delta (p,t)'}  \!\!    \(\frac{1}{\mu(B)}\int_Bw(x)^{\Delta (p,t)}d\mu\)^\frac{1}{\Delta (p,t)}\\
&\leq c_{1,\Delta (p,t)'}\|f\|_X[w]_{RH_{\Delta (p,t)}}\leq c_{1,\Delta (p,t)'} K(p,t)\|f\|_X.
\end{align*}
It follows that $b_{1,w}\leq K(p,t)c_{1,\Delta (p,t)'}$ for all $w$ as specified above and hence
 $$\Psi_p(t)\leq c_{1,\Delta (p,t)'} K(p,t).$$
\end{proof}

\begin{remark}
{\rm
We compare Proposition \ref{c:Psiestimate} applied to the traditional John-Nirenberg $BMO$ space to some estimates proved in \cite[Theorem 1.19]{HP}, which is one of the few articles we know of that track the constants for such inequalities.  The authors in \cite{HP}  proved that $\|f\|_{BMO_w}\leq c[w]_{A_\infty}'\|f\|_{BMO}$ for some dimensional constant $c>0$, where 
$$[w]_{A_\infty}'=\sup_{Q\in\mathcal Q}\frac{1}{w(Q)}\int_Q\mathcal M(\chi_Qw)(x)dx$$
and $\mathcal M$ is the standard Hardy-Littlewood maximal operator.  They showed that this constant is sharp in terms of the power on the weight character, in the sense that one cannot obtain this estimate with $([w]_{A_\infty}')^\epsilon$ in place of $[w]_{A_\infty}'$ for any $0<\epsilon<1$ (in fact, they showed something slightly better).  In terms of our notation, this says that $b_{1,w}\leq c[w]_{A_\infty}'$.
In this setting, we apply Proposition \ref{c:Psiestimate} with $\Delta(p,t)=1+\frac{1}{2^{n+1}t-1}$ and $K(p,t)=2$ as in Remark \ref{Delta+K}, which provides the estimate $\Psi_p(t)\lesssim t$ for all $1<p<\infty$ and $t\geq1$, since it is known that $c_{1,p}\lesssim p$ for all $1<p<\infty$.  Then for any $w\in A_p$, it follows that $b_{1,w}\leq\Psi_p([w]_{A_p})\lesssim [w]_{A_p}$.  We fail to recover the $A_\infty$ constant of \cite[Theorem 1.19]{HP}, but we do recover the linear dependence on the power of the weight constant.  Since we have not considered any $A_\infty$ constants in this work, aside from the current remark, this is the best possible result for Proposition \ref{c:Psiestimate}.  We will see in Section 5 other applications of Proposition \ref{c:Psiestimate} in other settings, where the results are new and obtain the same linear dependence on the $A_p$ weight character.
}
\end{remark}

\begin{remark}\label{r:Xp=X}
{\rm
The somehow artificially imposed assumption $X(\mathcal B,\Lambda)=X^p(\mathcal B,\Lambda)$ for $0<p<1$ in the hypotheses of the above results can be eliminated in some situations.  One of them is when $\mathfrak F=X(\mathcal B,\Lambda)$ in Theorem \ref{smallnecessity} or 
 $\mathfrak F=X_{w_0}(\mathcal B,\Lambda)$ in Theorem \ref{necessity}. This 
immediately forces $X(\mathcal B,\Lambda)=X^p(\mathcal B,\Lambda)$, respectively  $X_{w_0}(\mathcal B,\Lambda)=X_{w_0}^p(\mathcal B,\Lambda)$, for $0<p\leq1$. Indeed, for any $w$, by definition $X_w^p(\mathcal B,\Lambda)\subset \mathfrak F$ while  $X_w(\mathcal B,\Lambda) \subset X_w^p(\mathcal B,\Lambda)$ always holds in the range $0<p\leq 1$. 

Alternatively, $X_w(\mathcal B,\Lambda)=X_w^p(\mathcal B,\Lambda)$ for $0<p\leq 1$ holds true if more structure on $X_w$ is assumed. More precisely, suppose that $\|\cdot\|_{X_w^p}$ is a quasi-norm, $X_w^p(\mathcal B,\Lambda)$ endowed with it is a quasi-Banach space for all $0<p\leq1$, and  $X_w(\mathcal B,\Lambda)$ is dense in $X_w^p(\mathcal B,\Lambda)$ with respect to $\|\cdot\|_{X_w^p}$ also for all $0<p<1$.  If this is the case and $X_w(\mathcal B,\Lambda)$ satisfies the $p$-invariance property for $(0,1]$, then clearly again $X_w(\mathcal B,\Lambda)=X_w^p(\mathcal B,\Lambda)$ for $0<p\leq 1$. 

Finally, in the classical $BMO$ context, very general conditions under which
 \begin{align*}
\sup_{Q}\frac{1}{|Q|}\int_Q h(|f(x)- f_Q|)\,dx<\infty
\end{align*}
for an appropriate function $h$ implies $f\in BMO$ were given by Str\"omberg \cite{stro}, Lo and Ruilin \cite{LR}, and  Shi and Torchinsky \cite{ST}. See also the more recent work of  Logunov et al in \cite{LSSVZ}. Hence, if $X=BMO$ and $\Lambda(f,Q)^p=h((|f(x)- f_Q|)$ for a certain such appropriate function, then one can conclude that $X^p=X$ for $0<p<1$. We will adapt some of these works to several applications we present in Section 5.
}
\end{remark}

\section{Necessary Conditions for Weight Invariance}

In this section, we provide a partial converse to the results from the previous one.  Although in applications we will obtain weight invariance from known $p$-invariance results, it is natural to ask whether the two concepts are actually equivalent. We indeed show that, essentially,  if we have weight invariance estimates depending only on the norm of the weights, then $p$-power John-Nirenberg properties also hold. Assumptions {\it A3} and {\it A4} and the function $\Psi_p$ defined in \eqref{Psi} play a pivotal role.  They allow us to perform  several computations to estimate $\|f\|_{X^p}$, and additional estimates for $\Psi_p$ impose the right control in terms of the weights.  In the end, we will be able to estimate $c_{1,p}$ using $\Psi_p$ as stated in equation \eqref{c1p} of Theorem \ref{sufficiency} below, which when combined with Proposition \ref{c:Psiestimate} provides a precise quantitative way to associate the $p$-invariance and weight invariance through the constants $c_{1,p}$ and $\Psi_p(t)$.

The following lemma is likely known in many settings. We include the computations just to show that it does not depend on any particular property of the measure,  the family of  sets $\mathcal B$, or its associated maximal function.

\begin{lemma}\label{lemmasufficiency}
If $u\in A_p(\mathcal B)\cap RH_\delta(\mathcal B)$ for some $1<p,\delta<\infty$, then $u^\delta\in A_q(\mathcal B)$ and $[u^\delta]_{A_q}\leq[u]_{RH_\delta}^\delta[u]_{A_p}^\delta$ where $q=1+\delta( p-1)$.
\end{lemma}

\begin{proof}
Let $u\in A_p(\mathcal B)\cap RH_\delta(\mathcal B)$ for some $1<\delta,p<\infty$, and define $q=1+\delta( p-1)$.  Note that with this selection we have $q'=\frac{p'}{\delta p}+1$, and $\frac{q'}{q} =\frac{p'}{\delta p}$.  Then it follows that
\begin{align*}
&\(\frac{1}{\mu(B)}\int_Bu(x)^\delta d\mu(x)\)\(\frac{1}{\mu(B)}\int_Bu(x)^{-\delta q'/q}d\mu(x)\)^{q/q'}\\
&\leq[u]_{RH_\delta}^\delta\(\frac{1}{\mu(B)}\int_Bu(x)\,d\mu(x)\)^\delta\(\frac{1}{\mu(B)}\int_Bu(x)^{- {p'/p}}d\mu(x)\)^{\delta p/p'}\\
& \leq[u]_{RH_\delta}^\delta[u]_{A_p}^\delta
\end{align*}
for all $B\in\mathcal B$.  Therefore $u^\delta\in A_q(\mathcal B)$ with $[u^\delta]_{A_q}\leq[u]_{RH_\delta}^\delta[u]_{A_p}^\delta$.
\end{proof}

The next theorem shows that if $X(\mathcal B,\Lambda)$ is weight invariant for the class $\mathcal W= A_p$ with constants controlling the ``norm" equivalences depending only on the characteristic of the weights, then $X(\mathcal B,\Lambda)$ is also $p$-invariant. More precisely,

\begin{theorem}\label{sufficiency}
Suppose $\mu$, $\mathcal B$, and $\mathcal M_{\mathcal B}$ satisfy the assumptions {\it A1}--{\it A4}.  Assume there exists $1<p_0<\infty$ such that $X(\mathcal B,\Lambda)$ satisfies the weight invariance property for $A_{p_0}(\mathcal B)$ and that $\Psi_p(2\|\mathcal M_\mathcal B\|_{L^p,L^p})$ is finite for all $1<p< p_0$.  Then $X(\mathcal B,\Lambda)$ satisfies the $p$-power John-Nirenberg property for the interval $(0,\infty)$.  Moreover, it also holds that 
\begin{align}
c_{1,p}\leq2 \Psi_{p'}(2\|\mathcal M_\mathcal B\|_{L^{p'},L^{p'}})<\infty\label{c1p}
\end{align}
for all  $1< p <\infty$.
\end{theorem}

\begin{proof}
We will first use a bootstrapping argument to prove that $c_{1,p}<\infty$ without proving the estimate asserted in \eqref{c1p}, as the constants in these initial arguments are difficult to track.  Once we know that $c_{1,p}$ is finite for all $1<p<\infty$, we can revisit and streamline the argument to obtain the estimate on \eqref{c1p}.  It should be noted that we cannot use the streamlined proof directly since it requires the a priori knowledge that $c_{1,p}<\infty$.

Assume $X(\mathcal B,\Lambda)$ satisfies the weight invariance property for $A_{p_0}(\mathcal B)$ for some $1<p_0<\infty$.  Choose $1<s<p_0$ small enough so that 

$$2\|\mathcal M_\mathcal B\|_{L^s,L^s}\geq \left( 1+ K (p_0, 2\|\mathcal M_\mathcal B\|_{L^{p_0},L^{p_0}}) \, 2 \|\mathcal M_{\mathcal B}\|_{L^{p_0},L^{p_0}}  \right)^{p_0},$$
which is possible by assumption {\it A3}.  Also fix two more parameters $1<r<s$ and $1<\delta<\min\(\frac{s-1}{r-1},p_0,\Delta\(p_0,2\|\mathcal M_\mathcal B\|_{L^{p_0},L^{p_0}}\)\)$.  

Define the $L^{p_0}$-adapted Rubio de Francia algorithm
\begin{align*}
\mathcal R^{({p_0})}g(x)=\sum_{k=0}^\infty\frac{\mathcal M_{\mathcal B}^kg(x)}{(2\|\mathcal M_{\mathcal B}\|_{L^{p_0},L^{p_0}})^k}.
\end{align*}
Here $\mathcal M_{\mathcal B}^0g(x)=|g(x)|$ and $\mathcal M_{\mathcal B}^kg$ is the $k$-fold iterated application of $\mathcal M_{\mathcal B}$ to a function $g$.  By assumption {\it A3}, we have $\|\mathcal R^{(p_0)}\|_{L^{p_0},L^{p_0}}\leq 2$.

Fix $f\in X(\mathcal B,\Lambda)$.  Let $B\in\mathcal B$, and $u(x)=\mathcal R^{(p_0)}[\Lambda(f,B)^{1/p_0}\cdot\chi_B](x)$.  Note that $\Lambda(f,B)$ is measurable and non-negative and 
$$\int_B\Lambda(f,B)(x)d\mu(x)\leq\mu(B)\|f\|_X<\infty,$$
and hence $\Lambda(f,B)^{1/p_0}\cdot\chi_{B}$ is an $L^{p_0}(\mu)$ function.  Therefore $\mathcal R^{(p_0)}[\Lambda(f,B)^{1/p_0}\cdot\chi_B]$ is well-defined as an $L^{p_0}(\mu)$ function.  Since $$\mathcal M_{\mathcal B}u(x)\leq 2\|\mathcal M_{\mathcal B}\|_{L^{p_0},L^{p_0}}u(x),$$ it follows by {\it A1} that $u\in A_1(\mathcal B)\subset A_r(\mathcal B)\subset A_{p_0}(\mathcal B)$ with $$[u]_{A_{p_0}}\leq[u]_{A_r}\leq[u]_{A_1}\leq2\|\mathcal M_{\mathcal B}\|_{L^{p_0},L^{p_0}}.$$ By assumption {\it A4}, $[u]_{A_{p_0}}\leq2\|\mathcal M_{\mathcal B}\|_{L^{p_0},L^{p_0}}$ implies that 
$$u\in RH_{\Delta (p_0,2\|\mathcal M_\mathcal B\|_{L^{p_0},L^{p_0}})}$$ with 
$$[u]_{RH_{\Delta (p_0,2\|\mathcal M_{\mathcal B}\|_{L^{p_0},L^{p_0}})}}\leq K(p_0,2\|\mathcal M_{\mathcal B}\|_{L^{p_0},L^{p_0}}).$$
With $1<\delta<\min(\frac{s-1}{r-1},p_0,\Delta(p_0,2\|\mathcal M_\mathcal B\|_{L^{p_0},L^{p_0}}))$ as specified above, define $v=u^\delta$.  Using Lemma~\ref{lemmasufficiency} and the fact that  our parameter selection implies $s>1+\delta(r-1)$, it follows that 
$v\in A_s(\mathcal B) \subset A_{p_0}(\mathcal B) $ 
and
\begin{align*}
[v]_{A_s} & \leq([u]_{RH_\delta}[u]_{A_r})^\delta\leq \(K(p_0, 2\|\mathcal M_{\mathcal B}\|_{L^{p_0},L^{p_0}}\) \, 2 \|\mathcal M_{\mathcal B}\|_{L^{p_0},L^{p_0}} )^\delta \\
 & \leq2\|\mathcal M_\mathcal B\|_{L^s,L^s}.
 \end{align*}
Then
\begin{align*}
&\(\frac{1}{\mu(B)}\int_B\Lambda(f,B)(x)^{1+1/p_0} d\mu(x)\)^\frac{1}{1+1/p_0}\\
& \hspace{2cm} \leq\(\frac{1}{\mu(B)}\int_B\Lambda(f,B)(x)^{1+\delta/p_0} d\mu(x)\)^\frac{1}{1+\delta/p_0}\\
& \hspace{2cm} =\(\frac{1}{\mu(B)}\int_B\Lambda(f,B)(x)\(\Lambda(f,B)(x)^\frac{1}{p_0}\)^{\delta} \chi_B(x) d\mu(x)\)^\frac{1}{1+\delta/p_0}\\
& \hspace{2cm} \leq \(\frac{1}{\mu(B)}\int_{B}\Lambda(f,B)(x)v(x)d\mu(x)\)^\frac{1}{1+\delta/p_0}\\
& \hspace{2cm} \leq\(\frac{v(B)}{\mu(B)}\|f\|_{X_v(\R^n)}\)^\frac{1}{1+\delta/p_0}.
\end{align*}
Since $\delta<p_0$, it also follows that
\begin{align*}
\frac{v(B)}{\mu(B)}& = \frac{1}{\mu(B)}\int_B \( \mathcal R^{(p_0)}[\Lambda(f,B)^{1/p_0}\cdot\chi_B](x)\)^{\delta}d\mu(x)\\
&\leq\(\frac{1}{\mu(B)}\int_B  \( \mathcal R^{(p_0)}[\Lambda(f,B)^{1/p_0}\cdot\chi_B](x)\)^{p_0} d\mu(x)\)^{\delta/p_0}\\
& \leq2^\delta\(\frac{1}{\mu(B)}\int_B\Lambda(f,B)(x)d\mu(x)\)^{\delta/p_0}\\
&\leq2^{\delta}\|f\|_{X}^{\delta/p_0}.
\end{align*}
It follows that
\begin{align*}
 \(\frac{1}{\mu(B)}\int_B\Lambda(f,B)(x)^{1+\delta/p_0} d\mu(x)\)^\frac{1}{1+\delta/p_0}&\leq\(\frac{v(B)}{\mu(B)}\|f\|_{X_v}\)^\frac{1}{1+\delta/p_0}\\
&\hspace{-1cm} \leq2^\frac{\delta}{1+\delta/p_0}\|f\|_{X}^\frac{\delta/p_0}{1+\delta/p_0}\|f\|_{X_v}^\frac{1}{1+\delta/p_0}\\
&\hspace{-1cm} \leq2^\frac{\delta}{1+\delta/p_0}b_{1,v}^\frac{1}{1+\delta/p_0}\|f\|_{X}\\
&\hspace{-1cm} \leq2^\frac{\delta}{1+\delta/p_0}\Psi_s(2\|\mathcal M_{\mathcal B}\|_{L^s,L^s})^\frac{1}{1+\delta/p_0}\|f\|_{X}.
\end{align*}
Therefore $X(\mathcal B,\Lambda)$ satisfies the $p$-power John-Nirenberg property on the interval $[1,1+1/p_0]$ and hence, by Lemma~\ref{l:smallpJN}, it does so on $(0,1+1/p_0]$ too.

We will now bootstrap this argument to show that $X(\mathcal B,\Lambda)$ satisfies the $p$-power John-Nirenberg inequality for the interval $(0,\infty)$.  We will do so by induction.

 Write $1+1/p_0 = (4p_0+4)/4p_0 $ and   assume that $X(\mathcal B,\Lambda)$ satisfies the $p$-power John-Nirenberg property on the interval $(0,\frac{4p_0+\ell-1}{4p_0}]$ for some integer 
 $\ell\geq 5$, and define the $L^{\ell'}$-adapted Rubio de Francia algorithm
\begin{align*}
\mathcal R^{(\ell')}g(x)=\sum_{k=0}^\infty\frac{\mathcal M_{\mathcal B}^kg(x)}{(2\|\mathcal M_{\mathcal B}\|_{L^{\ell'},L^{\ell'}})^k}.
\end{align*}
Here, $\ell'=\frac{\ell}{\ell-1}$ is the H\"older conjugate of $\ell$ and now $\|\mathcal R^{(\ell')}\|_{L^{\ell'},L^{\ell'}}\leq~2$.  Fix $1<s<\ell'$ small enough so that
$$2\|\mathcal M_\mathcal B\|_{L^{s},L^{s}}\geq \(1+K(\ell', 2\|\mathcal M_\mathcal B\|_{L^{\ell'},L^{\ell'}})\,  2\|\mathcal M_\mathcal B\|_{L^{\ell'},L^{\ell'}} \)^{\ell'},$$
and, like in the initial step, fix $1<r<s$ and 
$$1<\delta_\ell<\min\(\frac{s-1}{r-1},\ell',\Delta(\ell',2\|\mathcal M_\mathcal B\|_{L^{\ell'},L^{\ell'}})\).$$
Fix $f\in X(\mathcal B,\Lambda)$ and $B\in\mathcal B$. 
Let 
$$u_\ell(x)=\mathcal R^{(\ell')}[\Lambda(f,B)^\frac{4p_0+\ell-1}{4p_0\,\ell'}\cdot\chi_B](x).$$
 Note that $\Lambda(f,B)^\frac{4p_0+\ell-1}{4p_0\,\ell'}\cdot\chi_B\in L^{\ell'}(\R^n)$ since, by the inductive hypothesis, 
 $$\|f\|_{X^\frac{4p_0+\ell-1}{4p_0}}\leq c_{1,\frac{4p_0+\ell-1}{4p_0}}\|f\|_{X}<\infty.$$
 It follows that 
 $$u_\ell\in A_1(\mathcal B)\subset  A_r(\mathcal B)\subset  A_{\ell'}(\mathcal B),$$ 
 with 
 $$
 [u_\ell]_{A_{\ell'}}\leq[u_\ell]_{A_r}\leq[u_\ell]_{A_1}\leq2\|\mathcal M_{\mathcal B}\|_{L^{\ell'},L^{\ell'}},
 $$
  and hence  $u_\ell\in RH_{\Delta (\ell',2\|\mathcal M_\mathcal B\|_{L^{\ell'},L^{\ell'}})}(\mathcal B)$, with 
  $$
  [u_\ell]_{RH_{\Delta (\ell',2\|\mathcal M_\mathcal B\|_{L^{\ell'},L^{\ell'}})}}\leq K(\ell',2\|\mathcal M_\mathcal B\|_{L^{\ell'},L^{\ell'}}).
  $$
   Define $v_\ell=u_\ell^{\delta_\ell}$, and using the same arguments as before, Lemma~\ref{lemmasufficiency} implies that $v_\ell\in A_s(\mathcal B)$ with
\begin{align*}
[v_\ell]_{A_s}\leq([u_\ell]_{RH_{\delta_\ell}}  [u_\ell]_{A_r})^{\delta_\ell} &\leq \(2K(\ell',2\|\mathcal M_\mathcal B\|_{L^{\ell'},L^{\ell'}})\,\|\mathcal M_{\mathcal B}\|_{L^{\ell'},L^{\ell'}}\)^{\delta_\ell}\\
&\leq2\|\mathcal M_\mathcal B\|_{L^s,L^s}.
\end{align*}
Using that $\ell^2 \leq (\ell-1)(4+ \ell -1)$ and  $\delta_\ell,p_0>1$, it follows that
\begin{align*}
&\frac{4p_0+\ell}{4p_0}\leq 1+\frac{\delta_\ell(4p_0+\ell-1)}{4p_0\ell'}.
\end{align*}
Then it also follows that
\begin{align*}
&  \(\frac{1}{\mu(B)}\int_B\Lambda(f,B)(x)^\frac{4p_0+\ell}{4p_0}d\mu(x)\)^\frac{4p_0}{4p_0+\ell}\\
&\hspace{.5cm} \leq\(\frac{1}{\mu(B)}\int_B\Lambda(f,B)(x)^{1+\frac{\delta_\ell(4p_0+\ell-1)}{4p_0\ell'}}d\mu(x)\)^\frac{4p_0\ell'}{4p_0\ell'+\delta_\ell(4p_0+\ell-1)}\\
&\hspace{.5cm} =\(\frac{1}{\mu(B)}\int_B\Lambda(f,B)(x)\(\Lambda(f,B)(x)^\frac{4p_0+\ell-1}{4p_0\ell'}\)^{\delta_\ell}d\mu(x)\)^\frac{4p_0\ell'}{4p_0\ell'+\delta_\ell(4p_0+\ell-1)}\\
&\hspace{.5cm} \leq\(\frac{1}{\mu(B)}\int_B\Lambda(f,B)(x)v_\ell(x)d\mu(x)\)^\frac{4p_0\ell'}{4p_0\ell'+\delta_\ell(4p_0+\ell-1)}\\
&\hspace{.5cm}\leq\(\frac{v_\ell(B)}{\mu(B)}\|f\|_{X_{v_\ell}}\)^\frac{4p_0\ell'}{4p_0\ell'+\delta_\ell(4p_0+\ell-1)}.
\end{align*}
Since $1<\delta_\ell<\ell'$, we also have that
\begin{align*}
 \frac{v_\ell(B)}{\mu(B)} & \leq\frac{1}{\mu(B)}\int_B\mathcal R^{(\ell')}[\Lambda(f,B)^\frac{4p_0+\ell-1}{4p_0\ell'}\cdot\chi_B](x)^{\delta_\ell} d\mu(x)\\
& \leq\(\frac{1}{\mu(B)}\int_B\mathcal R^{(\ell')}[\Lambda(f,B)^\frac{4p_0+\ell-1}{4p_0\ell'}\cdot\chi_B](x)^{\ell'} d\mu(x)\)^\frac{\delta_\ell}{\ell'}\\
&\leq2^{\delta_\ell}\(\frac{1}{\mu(B)}\int_B\Lambda(f,B)(x)^\frac{4p_0+\ell-1}{4p_0}d\mu(x)\)^\frac{\delta_\ell}{\ell'} \\
& \leq2^{\delta_\ell}\|f\|_{X^\frac{4p_0+\ell-1}{4p_0}}^\frac{\delta_\ell(4p_0+\ell-1)} {4p_0\ell'}
 \less \|f\|_{X}^\frac{\delta_\ell(4p_0+\ell-1)}{4p_0\ell'},
\end{align*}
where we use that $\|f\|_{X^\frac{4p_0+\ell-1}{4p_0}}\less\|f\|_{X}$ by the inductive hypothesis.  Combining the above computations we obtain
\begin{align*}
& \(\frac{1}{\mu(B)}\int_B\Lambda(f,B)(x)^\frac{4p_0+\ell}{4p_0}d\mu(x)\)^\frac{4p_0}{4p_0+\ell}\\
& \hspace{1cm} \less\( \|f\|_{X}^\frac{\delta_\ell(4p_0+\ell-1)}{4p_0\ell'}\|f\|_{X_{v_\ell}}\)^\frac{4p_0\ell'}{4p_0\ell'+\delta_\ell(4p_0+\ell-1)}\\
& \hspace{1cm} \lesssim b_{1,v_\ell}^\frac{4p_0\ell'}{4p_0\ell'+\delta_\ell(4p_0+\ell-1)}\|f\|_{X}\\
&  \hspace{1cm} \lesssim  \Psi_s(2\|\mathcal M_\mathcal B\|_{L^s,L^s})^\frac{4p_0\ell'}{4p_0\ell'+\delta_\ell(4p_0+\ell-1)}\|f\|_{X}.
\end{align*}
Therefore the $p$-power John-Nirenberg property on $(0,\frac{4p_0+\ell-1}{4p_0}]$ for $X(\mathcal B,\Lambda)$ implies the $p$-power John-Nirenberg property on $(0,\frac{4p_0+\ell}{4p_0}]$ for $X(\mathcal B,\Lambda)$.  By induction, $X(\mathcal B,\Lambda)$ satisfies the $p$-power John-Nirenberg property for $(0,\infty)$.

Now that we know $c_{1,p}<\infty$ for all $1<p<\infty$, we can obtain the better estimate given in \eqref{c1p}.   In fact, by Proposition \ref{c:Psiestimate} the assumption on $\Psi_{p}$ in the theorem hypotheses improves now to $\Psi_{p}(2\|\mathcal M_\mathcal B\|_{L^{p},L^{p}})<\infty$, for all $1<p<\infty$. Suppose $f\in X(\mathcal B,\Lambda)$, and let $B\in\mathcal B$.  Fix $p$, $1<p<\infty$, and define $u(x)=\mathcal R^{(p')}[\Lambda(f,B)^{p-1}\cdot\chi_{B}](x)$, which is well defined since 
 $c_{1,p}<~\infty$ implies $\Lambda(f,B)^{p-1} \chi_B \in L^{p'}(\mu)$.  Here $\mathcal R^{p'}$ is the $L^{p'}$-adapted Rubio de Francia algorithm, similar to before.  It follows that $u\in A_1(\mathcal B)\subset A_{p'}(\mathcal B)$ with $[u]_{A_{p'}}\leq[u]_{A_1}\leq2\|\mathcal M_{\mathcal B}\|_{L^{p'},L^{p'}}$.  Then we have
\begin{align*}
\frac{1}{\mu(B)}\int_B&\Lambda(f,B)(x)^pd\mu(x)\leq\frac{1}{\mu(B)}\int_B\Lambda(f,B)(x)u(x)d\mu(x)\\
& \leq\frac{u(B)}{\mu(B)}\|f\|_{X_u}\leq b_{1,u}\frac{u(B)}{\mu(B)}\|f\|_{X}\leq \Psi_{p'}(2\|\mathcal M_\mathcal B\|_{L^{p'},L^{p'}})\frac{u(B)}{\mu(B)}\|f\|_{X}
\end{align*}
since $[u]_{A_{p'}}\leq2\|\mathcal M_\mathcal B\|_{L^{p'},L^{p'}}$.  We note that
\begin{align*}
\frac{u(B)}{\mu(B)}&\leq\(\frac{1}{\mu(B)}\int_B\mathcal R^{(p')}[\Lambda(f,B)^{p-1}\cdot\chi_B](x)^{p'}dx\)^\frac{1}{p'}\\
& \leq2\(\frac{1}{\mu(B)}\int_B\Lambda(f,B)^pdx\)^\frac{1}{p'}\leq2\|f\|_{X^p}^{p-1}.
\end{align*}
Then we have
\begin{align*}
\|f\|_{X^p}^p\leq 2\Psi_{p'}(2\|\mathcal M_\mathcal B\|_{L^{p'},L^{p'}})\|f\|_{X^p}^{p-1}\|f\|_{X}.
\end{align*}
Rearranging terms, we obtain $\|f\|_{X^p}\leq 2\Psi_{p'}(2\|\mathcal M_\mathcal B\|_{L^{p'},L^{p'}})\|f\|_{X}$ for all $f\in X(\mathcal B,\Lambda)$.  Therefore $c_{1,p}\leq 2\,\Psi_{p'}(2\|\mathcal M_\mathcal B\|_{L^{p'},L^{p'}})$.
\end{proof}

\section{Applications}

We present several applications including the theorems stated in the introduction. We repeat the statement of those theorems for the readers convenience and because we present some of them in a more general form.

Since we will consider many different measures in several different settings, we find it convenient to change for this section some of the the notation involving measures of sets and the corresponding averages of functions that we have been using so far. For example, when the underlying measure $\mu$ used to define an oscillation space $X(\mathcal B,\mu, \mathfrak F, \Lambda)$ is not the Lebesgue measure in $\rn$ and $w$ is a weight with respect to $\mu$, we will now write 
$$
\mu_w(B) =  \int_B  w(x) d\mu(x),
$$
instead of $w(B)$. However we will keep the latter notation in the Lebesgue setting. The precise meaning of other quantities will be consistent within each of the subsections and will be specified therein. 

In several of the results we will present, we will verify the condition $X^p=X$ for $0<p<1$ that appears in Theorems \ref{smallnecessity} and \ref{necessity}.  We were not able to find in the literature a proof of this property for each situation, but  we can adapt the techniques from \cite{LR}  for  our applications (our arguments are similar to the  proof of \cite[Proposition 2]{LR} except that we use the $p$-power property in place of the traditional John-Nirenberg exponential one).  Rather than rewriting the same argument for every application, we present the argument once in Lemma \ref{BMOp=BMO} in the terminology of our $X(\mathcal B,\Lambda)$ spaces.  We should also note that Lemma \ref{BMOp=BMO} is a version of Lemma \ref{l:smallpJN} with more structure imposed on $\Lambda$ and $X$, and hence we are able to conclude something stronger that addresses the technicalities arising in the $X^p=X$ conditions.

\begin{lemma}\label{BMOp=BMO}
Let   $X(\mathcal B,\mu, \mathfrak F, \Lambda)$  be so that $\mathfrak F\subset L^1_{loc}(\mu)$ and $\Lambda(f,B)(x)=|f(x)-f_B|$, where $f_B$ denotes the average of $f$ over $B$ with respect to $\mu$.  Then the following properties hold.

a) For all for all $f\in\mathfrak F$ 
\begin{align}\label{infequiv}
\|f\|_X\approx\sup_{B\in\mathcal B} \, \inf_{c\in\C}\frac{1}{\mu(B)}\int_B|f(x)-c|d\mu(x).
\end{align}

b) If in addition $X(\mathcal B,\Lambda)$ satisfies the $p$-invariance property for $[1,p_0)$ for some $1<p_0<\infty$, then $X^p(\mathcal B,\Lambda)=X(\mathcal B,\Lambda)$ for all $0<p<1$ and $X$ satisfies the $p$-invariance property for $(0,p_0)$
\end{lemma}

\begin{proof}
The proof of part a) is well-known.  Simply note that for any complex number $c$ 
$$\Lambda(f,B)(x)=|f(x)-f_B| \leq |f(x)-c| + |f_B-c|.$$
Taken then the average over $B$, followed by the infimum in $c$, we easily obtain 
$$\|f\|_X \leq 2 \sup_{B\in\mathcal B} \, \inf_{c\in\C}\frac{1}{\mu(B)}\int_B|f(x)-c|d\mu(x).$$
The reverse inequality is of course trivial.

Our first step to prove part b) is to show that $f\in X^p$ implies $\Lambda(f,B_0)^p\in X$ for all $B_0\in\mathcal B$ and $0<p<1$.  So fix $0< p<1$, $f\in X^p$, and $B_0\in\mathcal B$.  For every $B\in\mathcal B$, and since $0<p<1$, we have
\begin{align*}
\int_B\left| \Lambda(f,B_0)(x)^p-|f_B-f_{B_0}|^p\right|d\mu(x)&\leq\int_B|f(x)-f_B|^pd\mu(x)\\
&\leq\mu(B)\|f\|_{X^p}^p.
\end{align*}
Hence from  \eqref{infequiv} it follows that $\Lambda(f,B_0)^p\in X$ and that $\|\Lambda(f,B_0)^p\|_{X}\less\|f\|_{X^p}^p$.

From here we use a bootstrapping argument to show $X^p\subset X$ for all $0<p<1$, ranging all the way down to $0$.  More precisely, we prove (by induction) that $X^{1/p_0^k}\subset X$ for all $k\in\N$.  For $f\in X^{1/p_0}$, we have
\begin{align*}
\int_{B_0}\Lambda(f,B_0)(x)d\mu(x)&\leq2^{p_0} \!\int_{B_0}\left|\Lambda(f,B_0)(x)^{1/p_0}-\(\Lambda(f,B_0)^{1/p_0}\)_{B_0}\right|^{p_0} \!\!d\mu(x)\\
&\hspace{.8cm}+2^{p_0}\mu(B_0)\(\Lambda(f,B_0)^{1/p_0}\)_{B_0}^{p_0}\\
&=2^{p_0}\int_{B_0}\Lambda(\Lambda(f,B_0)(x)^{1/p_0},B_0)(x)^{p_0}d\mu(x)\\
&\hspace{.8cm}+2^{p_0}\mu(B_0)\(\frac{1}{\mu(B_0)}\int_{B_0}\Lambda(f,B_0)(x)^{1/p_0}d\mu(x)\)^{p_0}\\
&\leq2^{p_0}\mu(B_0)\|\Lambda(f,B_0)^{1/p_0}\|_{X^{p_0}}^{p_0}+2^{p_0}\mu(B_0)\|f\|_{X^{1/p_0}}\\
&\less c_{1,p_0}^{p_0}\mu(B_0)\|f\|_{X^{1/p_0}}.
\end{align*}
Note here that $c_{1,p_0}<\infty$ by assumption, and we used that $\Lambda(f,B_0)^{1/p_0}\in X$ implies 
$$\|\Lambda(f,B_0)^{1/p_0}\|_{X^{p_0}}\leq c_{1,p_0}\|\Lambda(f,B_0)^{1/p_0}\|_X\less c_{1,p_0}\|f\|_{X^{1/p_0}}^{1/p_0}.$$
Therefore $X^{1/p_0}\subset X$.  Now assume that $X^{1/p_0^k}\subset X$ holds for a given $k\geq1$.  Then for $f\in X^{1/p_0^{k+1}}$, by a similar argument to the $k=1$ case, we have
\begin{align*}
&\int_{B_0}\Lambda(f,B_0)(x)^{1/p_0^k}d\mu(x)\\
&\hspace{1.8cm}\leq2^{p_0}\int_{B_0}\left|\Lambda(f,B_0)(x)^{1/p_0^{k+1}}-\(\Lambda(f,B_0)^{1/p_0^{k+1}}\)_{B_0}\right|^{p_0}d\mu(x)\\
&\hspace{4.8cm}+2^{p_0}\mu(B_0)\(\Lambda(f,B_0)^{1/p_0^{k+1}}\)_{B_0}^{p_0}\\
&\hspace{1.8cm}\leq2^{p_0}\mu(B_0)\|\Lambda(f,B)^{1/p_0^{k+1}}\|_{X^{p_0}}^{p_0}+2^{p_0}\mu(B_0)\|f\|_{X^{1/p_0^{k+1}}}^{1/p_0^k}\\
&\hspace{1.8cm}\less\mu(B_0)c_{1,p_0}^{p_0}\|f\|_{X^{1/p_0^{k+1}}}^{1/p_0^k}.
\end{align*}
By induction, it follows that $X^{1/p_0^{k+1}}\subset X^{1/p_0^{k}} \subset X$ for all $k\in\N$.  Since $p_0>1$ so that $1/p_0^k\rightarrow0$ as $k\rightarrow\infty$, this is sufficient to prove that $X^p\subset X$ for all $0<p<1$.
\end{proof}

Lemma \ref{BMOp=BMO} will play a crucial role in the first three applications we shall present. In them the underlying metric spaces are, respectively, $\rn$ with the Lebesgue measure, an abstract space of homogeneous type, and $\rn$ with a non-necessarily doubling measure. In all three cases we will prove that the natural $BMO$ space of each context can be characterized also by an appropriate oscillation space defined using a pair of $A_\infty$ weights. Of course, the case of $\rn$ with the Lebesgue measure (and the Euclidean distance) is a particular case of the other two applications, but we chose to present a proof in this context to show how it relates to results in \cite{MW,stro,LR,ST} and for clarity in the exposition.  The proofs for all of the first three applications are also almost identical. Therefore, for spaces of homogeneous type and for non-doubling measures on $\rn$, we provide details about some needed estimates, but leave the computations  completely analogous to the classical case to the reader.  The proof of Theorem \ref{littlebmo}  also uses very similar arguments, and hence we omit many of the details as well.

\subsection{Weight invariance for the John-Nirenberg BMO space} The following extension of the weight invariant result in \cite{MW} holds.

\begin{theorem}\label{wbmo}
Let $0<p<\infty$, $w,v\in A_\infty$, and $f$ be a complex-valued function on $\R^n$. Then $f$ is in $L^1_{loc}(v)$ and satisfies
\begin{align*}
\sup_{Q\in\mathcal Q}\(\frac{1}{w(Q)}\int_Q|f(x)-f_{v,Q}|^pw(x)dx\)^\frac{1}{p} < \infty,
\end{align*}
where $f_{v,Q}=\frac{1}{v(Q)}\int_Qf(x)v(x)dx$ and $w(Q)=\int_Qw(x)dx$,
if and only if $f$ is in $BMO$. In such a case
\begin{align}\label{solobmo}
\|f\|_{BMO}\approx\sup_{Q\in\mathcal Q}\(\frac{1}{w(Q)}\int_Q|f(x)-f_{v,Q}|^pw(x)dx\)^\frac{1}{p}.
\end{align}
\end{theorem}

\begin{proof}
Set $\mathcal B=\mathcal Q$ and $\mu$ to be the Lebesgue measure.  For $v\in A_\infty$, we also set the notation $\mathcal F_v=L^1_{loc}(v)$, $\Lambda_v(f,Q)(x)=|f(x)-f_{v,Q}|$, and $X(\mathcal B,\mathfrak F_v,\mu,\Lambda_v)=X(\mathcal B,\Lambda_v)$.  With this notation, $X(\mathcal B,\mathfrak F_1,\mu,\Lambda_1)=X(\mathcal B,\Lambda_1)=BMO$ by definition and $\|f\|_{X_w^p(\mathcal B,\Lambda_v)}$ is the term appearing on the right hand side of \eqref{solobmo}.  So it is sufficient to show that $X_w^p(\mathcal B,\Lambda_v)\approx X(\mathcal B,\Lambda_1)$ for all $0<p<\infty$ and $w,v\in A_\infty$.

Muckenhoupt and Wheeden proved in \cite{MW}  the John-Nirenberg inequality  for $X_v(\mathcal B,\Lambda_v)$. More precisely, there exist constants $c_1,c_2>0$ (depending only on the dimension $n$) such that
$$
v(\{x\in Q:|f(x)-f_{v,Q}|>\lambda\})\leq c_1 v(Q)e^{-\frac{c_2\lambda}{\|f\|_{X_v(\mathcal B,\Lambda_v)}}},
$$
for all $f\in X_v(\mathcal B,\Lambda_v)$, $Q\in\mathcal B$, and $\lambda>0$. It follows that $X_v(\mathcal B,\Lambda_v)$ satisfies the $p$-invariance property for $[1,\infty)$ for all $v\in A_\infty$.

We complete this proof in three steps for any $w,v\in A_\infty$: 
\begin{itemize}
\item[1.]  $X^p_w(\mathcal B,\Lambda_v)\approx X_v^q(\mathcal B,\Lambda_v)$ for all $1\leq p<\infty$ and $0<q<\infty$;
\item[2.] $X_v(\mathcal B,\Lambda_v)\approx X(\mathcal B,\Lambda_1)$;
\item[3.] $X_w^p(\mathcal B,\Lambda_v)\approx X_v(\mathcal B,\Lambda_v)$ for all $0<p<1$
\end{itemize}

\medskip

\noindent{\it Step 1.}  Since $X_v(\mathcal B,\Lambda_v)$ satisfies the $p$-power John-Nirenberg property for $[1,\infty)$ and any $v\in A_\infty$, then it also satisfies by Lemma~\ref{BMOp=BMO} the $p$-invariance property on $(0,\infty)$ and $X_v(\mathcal B,\Lambda_v)\approx X_v^p(\mathcal B,\Lambda_v)$ for all such $p$.  Applying Theorem \ref{necessity} completes the proof statement 1.

Note that $X_w(\mathcal B,\Lambda_v)$ satisfies the $p$-invariant property for all of $(0,\infty)$, but we cannot yet conclude that the $X_w^p(\mathcal B,\Lambda_v)$ coincide for $0<p<1$ (see Remark \ref{pinvremark} for a discussion of this subtle point).  Note also, that we cannot apply Lemma \ref{BMOp=BMO} to $X_w^p(\mathcal B,\Lambda_v)$ since the weight $w$ does not match the weight on the oscillation functional $\Lambda_v$.

\medskip

\noindent{\it Step 2.}  For any $f\in X_v(\mathcal B,\Lambda_v)$, it follows from what we proved in {\it Step 1} that $f\in L^1_{loc}$ (since $f\in X(\mathcal B,\Lambda_v)$).  Furthermore, we have
\begin{align*}
\|f\|_{X(\mathcal B,\Lambda_1)}&\leq2\sup_Q\inf_{c\in\C}\frac{1}{|Q|}\int_Q|f(x)-c|dx\leq 2\|f\|_{X(\mathcal B,\Lambda_v)}\approx\|f\|_{X_v(\mathcal B,\Lambda_v)},
\end{align*}
where the last estimate also follows from {\it Step 1} with $p=q=1$ and $w=1$.  Similarly, $f\in X(\mathcal B,\Lambda_1)$ implies $f\in L^1_{loc}(v)$ and
\begin{align*}
\|f\|_{X_v(\mathcal B,\Lambda_v)}&\leq 2\|f\|_{X_v(\mathcal B,\Lambda_1)}\approx\|f\|_{X(\mathcal B,\Lambda_1)}.
\end{align*}
So statement 2 holds as well.

\medskip

\noindent{\it Step 3.}  By what we proved in {\it Step 1} with $p=q=1$, and using H\"older's inequality, it follows that 
$$X_v(\mathcal B,\Lambda_v)\approx X_w(\mathcal B,\Lambda_v)\subset X_w^p(\mathcal B,\Lambda_v)$$
holds for $0<p<1$, with $\|f\|_{X_w^p(\mathcal B,\Lambda_v)}\less\|f\|_{X_v(\mathcal B,\Lambda_v)}$ for $f\in L^1_{loc}(v)$.  Conversely, for $f\in X_w^p(\mathcal B,\Lambda_v)$ with $0<p\leq1$, set $\epsilon=\frac{p}{q\delta'}$ where $1<q,\delta<\infty$ are such that $w\in A_q$ and $v\in RH_\delta$.  Using what is by now a familiar argument (see the proof of Theorem \ref{necessity}), it follows that
\begin{align*}
&\frac{1}{v(Q)}\int_Q|f(x)-f_{v,Q}|^\epsilon v(x)dx\\
&\hspace{1cm}\leq [v]_{RH_\delta}[w]_{A_q}^{1/(q\delta')}\(\frac{1}{w(Q)}\int_Q|f(x)-f_{v,Q}|^pw(x)dx\)^{1/(\delta'q)}.
\end{align*}
Hence $X_w^p(\mathcal B,\Lambda_v)\subset X_v^\epsilon(\mathcal B,\Lambda_v)$ with $\|f\|_{X_v^\epsilon(\mathcal B,\Lambda_v)}\less\|f\|_{X_w^p(\mathcal B,\Lambda_v)}$ for all $f\in L^1_{loc}(v)$.  Using {\it Step 1} again, it also follows that $X_v^\epsilon(\mathcal B,\Lambda_v)\approx X_v(\mathcal B,\Lambda_v)$.  Therefore {\it Step 3}, and hence of the proof of Theorem \ref{wbmo}, is complete.
\end{proof}

\begin{remark}{\rm
A particular consequence of the estimate in {\it Step 1} is that $X_v^p(\mathcal B,\Lambda_v)\approx X_v(\mathcal B,\Lambda_v)$ for $0<p<1$, which we proved by applying Lemma~\ref{BMOp=BMO}, but this was already known in several situations; see  for example \cite{stro,LR,ST}. In the classical John-Nirenberg $BMO$ setting, the statement in {\it Step 2} is exactly the weight invariance of $BMO$ proved by Muckenhoupt and Wheeden in \cite{MW}.  We are not aware of any result in the form of the estimate in {\it Step 3} before this article.}
\end{remark}

Notice that the proofs of {\it Steps 1-3} in Theorem \ref{wbmo} depend only on the following facts: {\it A2} holds, $X_v(\mathcal B,\Lambda_v)$ satisfies the $p$-power John-Nirenberg property for $[1,\infty)$ for all $v\in A_\infty$, and that the oscillation functional is of the form $|f(x)-f_{\mu,Q}|$ for some measure $\mu$.  These are the properties we will verify to apply the same scheme of proof in the applications in the next two subsections.

\begin{remark} 
{\rm
A particular case of Theorem \ref{wbmo} is of course
\begin{equation}\label{ho}
\|f\|_{BMO}\approx \sup_{Q\in\mathcal Q}\(\frac{1}{w(Q)}\int_Q|f(x)-f_{Q}|^pw(x)dx\)^\frac{1}{p}.
\end{equation}
Using a notion of $p$-convexity for $1\leq p< \infty$, Ho \cite[Theorem 3.1]{H} proved \eqref{ho} when $w\in A_p$. Note that our theorem allows us to use $|f(x)-f_{v,Q}|$ in place of $|f(x)-f_Q|$ for any $v \in A_\infty$. Moreover, our methods allows us to prove  \eqref{ho} in the context of spaces of homogeneous type, recovering in the next subsection a particular case of a result of Franchi, P\'erez and Wheeden \cite[Theorem 3.1]{ FPW}, but allowing again the use of two weights to characterize $BMO$.
}
\end{remark}

\subsection{Weighted BMO on spaces of homogeneous type}

Let $(S,d,\mu)$ be a space of homogeneous type in the sense of Coifman and Weiss.  Let $\mathcal B$ be the collection of all $d$-balls $B$ in $S$, of the form 
$$B=B_d(x,r)=\{y\in S:d(x,y)<r\}.$$
 Define  for $w\in A_\infty$,
$$BMO_{w}(S,\mu):= \{ f\in L^1_{loc}(\mu_w): \|f\|_{BMO_{\mu_w}(S)}<\infty\},$$
where  
$$
\|f\|_{BMO_{\mu_w}(S)}:=\sup_{B\in\mathcal B}\frac{1}{\mu_w(B)}\int_B|f(x)-f_{\mu_w,B}|w(x)d\mu(x)
$$
with 
$$\mu_w(B) =\int_B w(x) d\mu(x),$$ and 
$$f_{\mu_w,B}=\frac{1}{w(B)}\int_Bf(x)w(x)d\mu(x).$$
  When $w\equiv 1$ we obtain the classical  John-Nirenberg space on a space of homogenous type and we simply write $BMO_\mu(S)$ in stead of 
$BMO_{\mu_1}(S)$. Likewise we will use the notation $f_{\mu,B}=f_{\mu_1,B}$.

\begin{theorem}\label{wbmoHT}
Let $0<p<\infty$, $v,w\in A_\infty(\mathcal B, \mu)$, and $f$ be a measurable complex-valued function on $S$. Then $f$ is in $L^1_{loc}(\mu_v)$ and satisfies
\begin{align*}
\sup_{B\in\mathcal B}\(\frac{1}{\mu_w(B)}\int_B|f(x)-f_{\mu_v,B}|^pd\mu_w(x)\)^{1/p} < \infty
\end{align*}
if and only if $f$ is in $BMO_\mu(S)$. In such a case,
\begin{align*}
\|f\|_{BMO_\mu(S)}&\approx\sup_{B\in\mathcal B}\(\frac{1}{\mu_w(B)}\int_B|f(x)-f_{\mu_v,B}|^pd\mu_w(x)\)^{1/p}.
\end{align*}
\end{theorem}

\begin{proof}
We first observe that {\it A2} holds in this situation (even sharp forms of the reverse H\"older inequality are known, as discussed in Remark \ref{Delta+K}). Note also that if $(S,d,\mu)$ is a space of homogeneous type and $w\in A_\infty(\mathcal B,\mu)$, it follows that $(S,d,\mu_w)$ is also a space of homogeneous type.  It is well-known that the John-Nirenberg inequality holds for spaces of homogeneous type. This fact has been reproved many times in the literature, but can actually be traced back to the the original work of Coifman and Weiss \cite[p. 694, fn. 22]{CW}, as observed in \cite{LR}, where the details are also presented.  Hence, there exist constants $c_1,c_2>0$ such that
\begin{align*}
\mu_w(\{x\in B:|f(x)-f_{\mu_w,B}|>\lambda\})\leq c_1\mu_w(B)e^{-c_2\lambda/\|f\|_{BMO_{\mu_w}}}
\end{align*}
for all $f\in BMO_{\mu_w}(S)$.  As explained before, this makes up all of the ingredients necessary to reproduce the proof of Theorem \ref{wbmo} in the setting of spaces of homogeneous type.  The details are left to the reader.
\end{proof}

\begin{remark}
{\rm
Recall from Remark \ref{Delta+K} that if we take $\Delta(p,t)=1+\frac{1}{\tau t}$ and $K(p,t)=C$, then the classes $A_p(\mathcal B,\mu)$, when $\mathcal B$ and $m$ are as in Theorem \ref{wbmoHT}, satisfy {\it A4}.  Then applying Proposition \ref{c:Psiestimate}, it follows that 
$$\Psi_p(t)\less c_{1,1+\tau t}\less t,$$
 and hence that $b_{1,w}\leq\Psi_p([w]_{A_p})\less[w]_{A_p}$ for any $w\in A_p$ with $1<p<\infty$.  In particular,
 $$\sup_{B\in\mathcal B}\frac{1}{\mu_w(B)}\int_B|f(x)-f_{\mu,B}|w(x)d\mu(x)\less[w]_{A_p}\|f\|_{BMO_\mu(S)}$$
for all $f\in BMO_\mu(S)$, $1<p<\infty$, and $w\in A_p(\mathcal B, \mu)$, where the suppressed constant does not depend on $f$, $p$, or $w$.
}
\end{remark}

\subsection{Weighted BMO with respect to non-doubling measures}

Let $\mu$ be a non-negative Radon measure on $\R^n$ (not necessarily doubling).  Define 
$$
BMO_\mu(\rn) :=\{ f\in L^1_{loc}(\rn,\mu): \|f\|_{BMO_\mu(\rn)}<\infty\},
$$
where
$$
\|f\|_{BMO_\mu(\rn)}:=\sup_{Q\in\mathcal Q}\frac{1}{\mu(Q)}\int_Q|f(x)-f_{\mu,Q}|d\mu(x)
$$
and 
$$f_{\mu,Q}=\frac{1}{\mu(Q)}\int_Qf(x)d\mu(x).$$
  Also, for any $w\in A_\infty(\mathcal Q, \mu)$ let 
$$\mu_w(Q)=\int_Qw(x)d\mu(x).$$

\begin{theorem}\label{ndbmo}
Let $\mu$ be a non-negative Radon measure on $\R^n$ such that $\mu(L)=0$ for any hyperplane $L$ orthogonal to one of the coordinate axes. Let $0<p<\infty$, $v,w\in A_\infty(\mathcal Q,\mu)$, and $f$ be a measurable complex-valued function on $\R^n$.  Then $f$ is in $L^1_{loc}(\mu_v)$ and satisfies
\begin{align*}
\sup_{Q\in\mathcal Q}\(\frac{1}{\mu_w(Q)}\int_Q|f(x)-f_{\mu_v,Q}|^pd\mu_w(x)\)^{1/p} <\infty,
\end{align*}
if and only if  $f$ is in $BMO_\mu(\rn)$.  In such a case,
\begin{align}\label{esta}
\|f\|_{BMO_\mu}&\approx \sup_{Q\in\mathcal Q}\(\frac{1}{\mu_w(Q)}\int_Q|f(x)-f_{\mu_v,Q}|^pd\mu_w(x)\)^{1/p}.
\end{align}
\end{theorem}

\begin{proof}
We note that {\it A2} holds in this situation as it was proved in \cite[Lemma 2.3]{OP}. Next, it was shown in \cite[Theorem 1]{MMNO} that there exist constants $c_1,c_2>0$ (depending only on the dimension $n$) such that
\begin{align}\label{ndjn}
\mu(\{x\in Q:|f(x)-f_{\mu,Q}|>\lambda\})\leq c_1\mu(Q)e^{-\frac{c_2\lambda}{\|f\|_{BMO_\mu}}}
\end{align}
for all $Q\in\mathcal B$ and $\lambda>0$.  Finally, note also that for any $w\in A_\infty(\mu)$, $\mu_w$ is absolutely continuous with respect to $\mu$.  So in particular, $\mu_w(L)=0$ for any hyperplane $L$ orthogonal to one of the coordinate axes.  Hence it follows that \eqref{ndjn} hold with $\mu_w$ in place of $\mu$.  Once again, this verifies everything needed to reproduce the proof of Theorem \ref{wbmo}. 
\end{proof}

\begin{remark} 
{\rm
Recall from Remark \ref{Delta+K} that if we take $\Delta(p,t)=1+\frac{1}{2^{p+1}B(n)t}$ and $K(p,t)=2$, then the classes $A_p(\mu)$, when $\mu$ is as in Theorem \ref{ndbmo}, satisfy {\it A4}.  Then applying Proposition \ref{c:Psiestimate}, it follows that 
\begin{align*}
\Psi_p(t)\leq 2c_{1,1+2^{p+1}B(n)t}\less 2^pB(n)t
\end{align*}
for $1<p<\infty$ and $t\geq1$.  Note that the linear estimate $c_{1,p}\less p$ holds for $1<p<\infty$ as a consequence of the John-Nirenberg inequality proved in \cite{MMNO}.  So for a particular $w\in A_p(\mu)$, we obtain 
$$b_{1,w}\leq\Psi_p([w]_{A_p})\less 2^p[w]_{A_p}.$$
Here we suppress the dependence on $B(n)$ since it is a dimensional constant.  In particular, 
$$\sup_{Q\in\mathcal Q}\frac{1}{w(Q)}\int_Q|f(x)-f_Q|w(x)dx\less 2^p[w]_{A_p}\|f\|_{BMO_\mu}$$
for any $f\in BMO_\mu$, $1<p<\infty$, and $w\in A_p(\mu)$, where the suppressed constant does not depend on $f$, $p$, or $w$.
}
\end{remark}

\subsection{Duals of weighted Hardy spaces}
Let $w$ be an $A_\infty$ weight in $\rn$ with respect to the Lebesgue measure. It is known that the dual of the Hardy space $H^1(w)$ (we will not need the definition of $H^1(w)$ here) is the space 
$$
  {BMO}_{*,w}:= \{ f\in L^1_{loc}(\rn): \|f\|_{BMO_{*,w}}<\infty\},
$$
taken modulo constants, where 
$$
\|f\|_{{BMO}_{*,w}}  :=  \sup_{Q\in\mathcal Q} \frac{1}{w(Q)}\int_Q|f(x)-f_Q| dx
$$
and $f_Q=\frac{1}{|Q|}\int_Qf(x)dx$. See \cite[Theorem II.4.4]{GC} for more information on $H^1(w)$ and a proof that its dual is $BMO_{*,w}$.  Our general setup easily gives the following invariance result in this context.

\begin{theorem}\label{t:tildeBMO}\label{hbmo} 
Let $w\in A_p$ for some $1<p<\infty$.  If $v\in RH_p( w)$, then
\begin{align}\label{notquite}
\|f\|_{{BMO}_{*,w}} &  \,\,\, \approx\sup_{Q\in\mathcal Q}\frac{1}{\rho(Q)}\int_Q|f(x)-f_Q| v(x)dx,
\end{align}
for all $f\in{{BMO}_{*,w}}$, where $\rho(Q)=\int_Q v(x)w(x)dx$.  Furthermore, if $v\in RH_\sigma(w)$ for some $\sigma>p$, then for all $0<r\leq p'/\sigma'$, we have
\begin{align*}
\|f\|_{{BMO}_{*,w}} &\approx\sup_{Q\in\mathcal Q}\(\frac{1}{\rho(Q)}\int_Q|f(x)-f_Q|^rw(x)^{1-r}v(x)dx\)^{1/r}
\end{align*}
for all $f\in{{BMO}_{*,w}}$.  

If $w\in A_1$, $v\in A_\infty(w\,dx)$, and $0< r<\infty$, then 
\begin{align*}
\|f\|_{{BMO}_{*,w}} & \approx\sup_{Q\in\mathcal Q}\(\frac{1}{\rho(Q)}\int_Q|f(x)-f_Q|^rw(x)^{ 1-r}v(x)dx\)^{1/r}
\end{align*}
for all $f\in{{BMO}_{*,w}}$.
\end{theorem}

\begin{proof}
First fix $w\in A_p$ for some $1<p<\infty$.  Define 
$\mathfrak F={{BMO}_{*,w}}$ and  $\Lambda(f,Q)(x)=|f(x)-f_Q|w(x)^{-1}$.  If we let 
$d\mu(x)=w(x)dx$, then it follows that $X(\mathcal Q,\mu,\Lambda)={{BMO}_{*,w}}$ since
\begin{align*}
\|f\|_X&=\sup_{Q\in\mathcal Q}\frac{1}{\mu(Q)}\int_Q\Lambda(f,Q)(x)d\mu(x)=\sup_{Q\in\mathcal Q}\frac{1}{w(Q)}\int_Q|f(x)-f_Q|dx.
\end{align*}
It was proved independently by Muckenhoupt and Wheeden \cite[Theorem 4]{MW} and  Garcia-Cuerva \cite[Theorem II.4.4]{GC}  that 
\begin{align*}
\sup_{Q\in\mathcal Q}\(\frac{1}{w(Q)}\int_Q|f(x)-f_Q|^rw(x)^{1-r}dx\)^{1/r}\less\|f\|_{{BMO}_{*,w}}
\end{align*}
for all $f\in {BMO}_{*,w}$ and $1\leq r\leq p'$.  In other words, $X(\mathcal Q,\mu, \Lambda)$ satisfies the John-Nirenberg $p$-power inequality for $[1,p']$.  Note that since we defined $\mathfrak F={{BMO}_{*,w}}$, it follows immediately that $X^p(\mathcal Q,\mu, \Lambda)=X(\mathcal Q,\mu, \Lambda)$ for all $0<p<1$, and Theorem \ref{smallnecessity} can be applied to $X(\mathcal Q,\mu, \Lambda)$.  The first part of Theorem \ref{smallnecessity} implies that for any $v\in RH_p(w)$,
\begin{align*}
\|f\|_{{BMO}_{*,w}} & =\|f\|_X\approx\|f\|_{X_v} =\sup_{Q\in\mathcal Q}\frac{1}{\rho(Q)}\int_Q|f(x)-f_Q|v(x)dx,
\end{align*}
for all $f\in{{BMO}_{*,w}}$, where $\rho(Q)=\int_Q v(x)w(x)dx$.  Furthermore, if $v\in RH_\sigma(w)$ for some $\sigma>p$, then the second part of Theorem \ref{smallnecessity} implies that $X_v(\mathcal Q,\Lambda)$ is $p$-invariant for the interval $(0,p'/\sigma']$.  That is,
\begin{align*}
\|f\|_{{BMO}_{*,w}} & =\|f\|_X\approx\|f\|_{X_v^r}\\
& =\sup_{Q\in\mathcal Q}\(\frac{1}{\rho(Q)}\int_Q|f(x)-f_Q|^rw(x)^{1-r}v(x)dx\)^\frac{1}{r}
\end{align*}
for all $f\in{{BMO}_{*,w}}$, $0<r\leq p'/\sigma'$, and $v\in RH_\sigma(w)$ when $\sigma>p$.  This completes the proof for $w\in A_p$ with $1<p<\infty$.  

Now let $w\in A_1$.  Since it is known that $A_1\subset\bigcap_{p>1}A_p$ and $A_\infty(w)=\bigcup_{p>1}RH_p(w)$, it follows from the estimates proved above that 
\begin{align*}
\|f\|_{{BMO}_{*,w}} & \approx\sup_{Q\in\mathcal Q}\(\frac{1}{\rho(Q)}\int_Q|f(x)-f_Q|^rw(x)^{1-r}v(x)dx\)^\frac{1}{r}
\end{align*}
for all $f\in{{BMO}_{*,w}}$, $v\in A_\infty(w)$, and $0<r<\infty$.
\end{proof}

\begin{remark} {\rm
Unlike the situation of the previous applications, we are not providing a full characterization of the space in question in Theorem \ref{hbmo}. 
We only show that \eqref{notquite} holds if $f \in BMO_{*,w}$. Though we suspect that if the righthand of \eqref{notquite} is finite then $f$ must be in 
$BMO_{*,w}$, our methods do not seem to be able to establish that.}
\end{remark}

\begin{remark}
{\rm
We note that Theorem \ref{hbmo} can be extended to weighted $BMO_{*,w}(S)$ spaces defined in the context of a space of homogeneous type $(S,d,\mu)$.  Indeed, the only additional information needed to reproduce the proof of Theorem \ref{hbmo} is that the $p$-power John-Nirenberg estimate for the space $BMO_{*,w}$, proved in \cite{MW,GC}, can be extended to $BMO_{*,w}(S)$.  A recent paper by Trong and Tung \cite{TT} does exactly this, and hence Theorem \ref{hbmo} can also be extended to a space of homogeneous type setting as an application of Theorem \ref{smallnecessity}.
}
\end{remark}

\begin{remark}
{\rm
The space $BMO_{*,w}$ was used by Bloom in \cite{B} to characterize the boundedness of commutators of the classical Hilbert transform between weighted Lebesgue spaces with different weights.  In a recent work Holmes, Lacey, and Wick \cite{HLW},  extended Bloom's result and  characterized the two-weight boundedness  of commutators of Calder\'on-Zygmund  operators in $\rn$.  In particular, \cite[Theorem 1.1]{HLW}  shows the following:  Let $1<p<\infty$, $\mu,\lambda\in A_p$, $\nu=(\mu\cdot\lambda^{-1})^{1/p}$, $b\in BMO_{*,\nu}$, and $T$ be a Calder\'on-Zygmund operator (see \cite{HLW} for the precise condition on $T$).  Then the commutator 
$$[b,T]f=bT(f) -T(bf)$$
is bounded from $L^p(\mu)$ into $L^p(\lambda)$ and
\begin{align*}
\|[b,T]\|_{L^p(\mu),L^p(\lambda)}\less\|b\|_{BMO_{*,\nu}}.
\end{align*}
It is not hard to see that $\nu\in A_2$, which was proved in \cite{HLW} as well.  It is also not hard to verify that $\nu=\mu^{1/p}\lambda^{-1/p}\in A_2$ implies $\nu^{-1}=\mu^{-1/p}\lambda^{1/p}\in RH_2(\nu)$.  Then by Theorem \ref{hbmo}, it follows that
\begin{align}\label{HLWest}
\|[b,T]\|_{L^p(\mu),L^p(\lambda)}\less\sup_{Q\in\mathcal Q}\frac{1}{|Q|}\int_Q|b(x)-b_Q|\nu(x)^{-1}dx.
\end{align}
Conversely, if $\nu\in RH_\delta$ for some $1<\delta<\infty$ and $b\in L^1_{loc}$ satisfies
\begin{align}\label{bvest}
\sup_{Q\in\mathcal Q}\frac{1}{|Q|}\int_Q|b(x)-b_Q|^{\delta'}\nu(x)^{-\delta'}dx<\infty,
\end{align}
then $b\in BMO_{*,\nu}$, and $[b,T]$ is bounded from $L^p(\mu)$ into $L^p(\lambda)$, and the estimate in \eqref{HLWest} holds.  To verify this, it is enough to check that
\begin{align*}
\int_Q|b(x)-b_Q|dx&\leq\(\int_Q|b(x)-b_Q|^{\delta'}\nu(x)^{-\delta'}dx\)^{1/\delta'}\(\int_Q\nu(x)^{\delta}dx\)^{1/\delta}\\
&\leq[\nu]_{RH_\delta}\nu(Q)\(\frac{1}{|Q|}\int_Q|b(x)-b_Q|^{\delta'}\nu(x)^{-\delta'}dx\)^{1/\delta'},
\end{align*}
and then apply \cite[Theorem 1.1]{HLW} and Theorem \ref{hbmo}.

If one imposes more on $\mu$ and/or $\lambda$, other estimates can be obtained by applying Theorem \ref{hbmo}.  For example, letting $p$, $\mu$, $\lambda$, $\nu$, $b$, and $T$ be as above, it follows that

\begin{itemize}
\item if $\lambda^{1/p}\in RH_2(\nu)$, then 
\begin{align*}
\|[b,T]\|_{L^p(\mu),L^p(\lambda)}\less\sup_{Q\in\mathcal Q}\frac{1}{\mu^{1/p}(Q)}\int_Q|b(x)-b_Q|\lambda(x)^{1/p}dx.
\end{align*}
\item if $\lambda\in RH_2(\nu)$, then
\begin{align*}
\|[b,T]\|_{L^p(\mu),L^p(\lambda)}\less\sup_{Q\in\mathcal Q}\frac{1}{\mu^{1/p}\lambda^{1/p'}(Q)}\int_Q|b(x)-b_Q|\lambda(x)dx.
\end{align*}
\item if $\mu^{-1/p}\in RH_2(\nu)$, then
\begin{align*}
\|[b,T]\|_{L^p(\mu),L^p(\lambda)}\less\sup_{Q\in\mathcal Q}\frac{1}{\lambda^{-1/p}(Q)}\int_Q|b(x)-b_Q|\mu(x)^{-1/p}dx.
\end{align*}
\end{itemize}

Furthermore, \cite[Theorem 1.2]{HLW} shows that when $1<p<\infty$, $\mu,\lambda\in A_p$, $\nu=(\mu\cdot\lambda^{-1})^{1/p}$, and $b\in BMO_{*,\nu}$, then $\|[b,R_j]\|_{L^p(\mu),L^p(\lambda)}\approx \|b\|_{BMO_{*,\nu}}$, where $R_j$ is the $j^{\text{th}}$ Riesz transform.  Hence, under the same assumptions, it follows that 
\begin{align}\label{Riesznorm}
\|[b,R_j]\|_{L^p(\mu),L^p(\lambda)}\approx \sup_{Q\in\mathcal Q}\frac{1}{|Q|}\int_Q|b(x)-b_Q|\nu(x)^{-1}dx.
\end{align}
Similarly, in each of the situations discussed above where $\|[b,T]\|_{L^p(\mu),L^p(\lambda)}$ is bounded by some maximal oscillation expression of $b$, one can obtain the same lower estimate for the commutator operator norm when $T=R_j$ is a Riesz transform.

Finally, if $\nu\in RH_\delta$ for some $1<\delta<\infty$ and $b\in L^1_{loc}$ satisfies \eqref{bvest}, then $[b,R_j]$ is bounded and \eqref{Riesznorm} holds.  Similar to the discussion above, this can be proved by noting that \eqref{bvest} implies $b\in BMO_{*,\nu}$ and then apply \cite[Theorem 1.2]{HLW}, followed by Theorem \ref{hbmo}.
}
\end{remark}

\subsection{BMO spaces associated to operators} 
There is a very extensive literature about $BMO$-type spaces defined by expression of the form 
$$
\sup_{Q\in\mathcal Q} \frac{1}{|Q|}\int_Q|f-S_{t_Q}f| dx,
$$
where $t_Q$ is a parameter associated to the cube $Q$ and $\{S_t\}_t>0$ is a semigroup, an approximation to the identity, or other appropriate family of operators. Moreover, further generalizations based on the approach in \cite{FPW}, with more general oscillating functional than $|f-S_{t_Q}f|$, and in weighted or different geometric context have been considered too.  We refer to the works of Duong-Yang \cite{DY1,DY}, Hofmann-Mayboroda \cite{HM}, Jimenez-del-Toro \cite{J-T}, Jimenez-del-Toro-Martell \cite{JM1}, Bernicot-Zhao \cite{BZ},   Bernicot-Martell \cite{BM}, and Bui-Dong \cite{BD}, to name a few.  

Although we will not pursue the analysis of these spaces here, it is easy to also represent them as $X(\mathcal B,\Lambda)$  for appropriate $\Lambda$'s. The interested reader could use exponential John-Nirembeg inequalities obtained in the mentioned works to obtain $p$-invariance properties and hence conclude using our methods weight invariance as well.

\subsection{Weighted little bmo}
Let $A_p(\R^{n_1}\times\R^{n_2})$ be the Muckenhoupt classes associated to rectangles of the form $R=Q_1\times Q_2$ for cubes $Q_1\subset\R^{n_1}$ and $Q_2\subset\R^{n_2}$.  We denote the class of all such rectangles by $\mathcal R=\mathcal R(\R^{n_1}\times\R^{n_2})$.  Let $w\in A_\infty(\R^{n_1}\times\R^{n_2})$.  For a  function $f\in L^1_{loc}(w)$, define
\begin{align*}
\|f\|_{bmo_w}=\sup_{R\in\mathcal R}\frac{1}{w(R)}\int_R|f(x)-f_{w,R}|w(x)dx,
\end{align*}
where $f_{w,R}=\frac{1}{w(R)}\int_Rf(x)dx$.  Let $bmo_w(\R^{n_1}\times\R^{n_2})$ be the collection of all $w$-locally integrable functions $f$ modulo constants such that $\|f\|_{bmo_w}<\infty$.  We include this notation of $bmo_w(\R^{n_1}\times\R^{n_2})$ to help with our presentation, but later, as a consequence of Theorem \ref{littlebmo}, we will see that the spaces $bmo_w(\R^{n_1}\times\R^{n_2})$ in fact coincide for all $w\in A_\infty(\R^{n_1}\times\R^{n_2})$.  This reproduces Muckenhoupt and Wheeden's result for $BMO$ in \cite{MW}.  To simplify the notation, we will write $bmo_w=bmo_w(\R^{n_1}\times\R^{n_2})$  with the convention that we are working with a fixed decomposition $\R^{n_1}\times\R^{n_2}$ of $\R^n$.  As usual, we also write $bmo$ in place of $bmo_1$ when $w=1$.

The proof of the next lemma is adapted from the proof of John-Nirenberg inequality in the lecture notes by Journ\'e \cite[pp. 31-32]{J}.

\begin{lemma}\label{l:JNbmo}
Let $w\in A_\infty(\R^{n_1}\times\R^{n_2})$.  There exists a constant $\eta>0$ such that 
\begin{align*}
\sup_{R\in\mathcal R}\frac{1}{w(R)}\int_Re^{|f(x)-f_{w,R}|/(\eta\|f\|_{bmo_w})}w(x)dx<\infty
\end{align*}
for all $f\in bmo_w$.
\end{lemma}

\begin{proof}
Let $w\in A_\infty(\R^{n_1}\times\R^{n_2})$ and $f\in bmo_w$ with norm $1$.  For $N\in\N$ and $\eta>0$, define
\begin{align*}
T_N(\eta)=\sup_{R\in\mathcal R}\frac{1}{w(R)}\int_Re^{\min(|f(x)-f_{w,R}|,N)/\eta}w(x)dx.
\end{align*}
Now fix $R=Q_1\times Q_2\in\mathcal R$, $\lambda>0$ (to be specified later), and choose $R_j=Q_{1,j}\times Q_{2,j}$ to be disjoint dyadic sub rectangles of $R$ such that $\ell(Q_{1,j})/\ell(Q_{2,j})=\ell(Q_{1})/\ell(Q_{2})$,
\begin{align*}
\lambda<\frac{1}{w(R_j)}\int_{R_j}|f(x)-f_{w,R}|w(x)dx\leq  D_w^2\lambda
\end{align*}
for all $j$, and $|f(x)-f_{w,R}|\leq\lambda$ almost everywhere on $R\backslash\bigcup_jR_j$.  Here $D_w$ denotes the doubling constant for $w$; i.e. $D_w$ the constant so that $w(2R)\leq D_ww(R)$ for all $R\in\mathcal R$.  (We use $ D_w^2$ here since for any dyadic rectangle $R$, $2^2R=4R$ engulfs its dyadic father).  Note that this selection implies that $|f_{w,R}-f_{w,R_j}|\leq  D_w^2\lambda$ for all $j$.  Then it follows that
\begin{align*}
\int_Re^{\min(|f(x)-f_{w,R}|,N)/\eta}w(x)dx&\leq\int_{R\backslash\bigcup_jR_j}e^{|f(x)-f_{w,R}|/\eta}w(x)dx\\
&\hspace{1cm}+\sum_j\int_{R_j}e^{\min(|f(x)-f_{w,R}|,N)/\eta}w(x)dx\\
&\hspace{-3.25cm}\leq e^{\lambda/\eta}w(R)+\sum_je^{|f_{w,R_j}-f_{w,R}|/\eta}\int_{R_j}e^{\min(|f(x)-f_{w,R_j}|,N)/\eta}w(x)dx\\
&\hspace{-3.25cm}\leq e^{\lambda/\eta}w(R)+\frac{e^{ D_w^2\lambda/\eta}T_N(\eta)}{\lambda}\sum_j\int_{R_j}|f(x)-f_{w,R}|w(x)dx\\
&\hspace{-3.25cm}\leq \(e^{\lambda/\eta}+\frac{e^{ D_w^2\lambda/\eta}T_N(\eta)}{\lambda}\)w(R).
\end{align*}
Dividing both sides by $w(R)$ and taking the supremum over all $R\in\mathcal R$, it follows that
\begin{align*}
T_N(\eta)&\leq e^{\lambda/\eta}+\frac{e^{ D_w^2\lambda/\eta}T_N(\eta)}{\lambda}.
\end{align*}
Finally fix $\lambda$ and $\eta$ so that $\frac{e^{ D_w^2\lambda/\eta}}{\lambda}\leq\frac{1}{2}$ (for example we can just set $\lambda=\eta=2e^{ D_w^2}$ to obtain $\frac{e^{ D_w^2\lambda/\eta}}{\lambda}=\frac{1}{2}$).  Then it follows that $T_N(\eta)\leq 2e^{\lambda/\eta}$ for all $N\in\N$ (note that truncating by $N$ in $T_N(\eta)$ assures us that $T_N(\eta)<\infty$).  Then for every $R\in\mathcal R$ and with $\lambda,\eta$ selected as above, it follows that by monotone convergence that
\begin{align*}
\frac{1}{w(R)}\int_Re^{|f(x)-f_{w,R}|/\eta}w(x)dx&\leq\lim_{N\rightarrow\infty}T_N(\eta)\leq2e^{\lambda/\eta}.
\end{align*}
This completes the proof.
\end{proof}

Some standard consequences of Lemma \ref{l:JNbmo} are that for any weight $w$ in $A_\infty(\R^{n_1}\times\R^{n_2})$ and $f$ in $bmo_w$, there exist constants $c_1,c_2>0$ such that
\begin{align*}
w(\{x\in R:|f(x)-f_{w,R}|>\lambda\})\leq c_1w(R)e^{-c_2\lambda/\|f\|_{bmo_w}}
\end{align*}
for all $R\in\mathcal R$ and $\lambda>0$, and that
\begin{align}\label{littlebmoineq}
\|f\|_{bmo_w^p}^p:=\sup_{R\in\mathcal R}\frac{1}{w(R)}\int_R|f(x)-f_{w,R}|^pw(x)dx\less\|f\|_{bmo_w}^p
\end{align}
for all $1\leq p<\infty$.

\begin{theorem}\label{littlebmo}
Let $0<p<\infty$, $w,v\in A_\infty(\R^{n_1}\times\R^{n_2})$, and $f$ be a complex-valued function on $\R^n$.  Then $f$ is in $L^1_{loc}(v)$ and satisfies
\begin{align*}
\sup_{R\in\mathcal R}\(\frac{1}{w(R)}\int_R|f(x)-f_{v,R}|^pw(x)dx\)^\frac{1}{p}<\infty
\end{align*}
if and only if $f$ is in $bmo$.  In such a case,
\begin{align}\label{solobmo2}
\|f\|_{bmo}\approx\sup_{R\in\mathcal R}\(\frac{1}{w(R)}\int_R|f(x)-f_{v,R}|^pw(x)dx\)^\frac{1}{p}.
\end{align}
\end{theorem}

\begin{proof}
Given that we have proved \eqref{littlebmoineq} and that {\it A2} holds for $A_\infty(\mathcal R)$ as was verified for example in \cite{GCRdF},  this proof follows exactly as in previous cases using Theorem \ref{necessity} and Lemma \ref{BMOp=BMO}.
\end{proof}

\begin{remark} 
{\rm
Recall from Remark \ref{Delta+K} that if we take $\Delta(p,t)=1+\frac{1}{2^{p+2}t}$ and $K(p,t)=2$, and $n_1=n_2=1$ as is the situation in \cite{LPR}, then the classes $A_p(\R\times\R)$ satisfies {\it A4}.  Applying Proposition \ref{c:Psiestimate}, it follows that 
\begin{align*}
\Psi_p(t)\leq 2c_{1,1+2^{p+2}t}\less 2^pt
\end{align*}
for $1<p<\infty$ and $t\geq1$.  Note that the linear estimate $c_{1,p}\less p$ holds for $1<p<\infty$ as a consequence of the John-Nirenberg inequality proved in Lemma \ref{l:JNbmo}.  So for $w\in A_p(\R\times\R)$, we set $t=[w]_{A_p}$ to obtain 
$$b_{1,w}\leq\Psi_p([w]_{A_p})\less 2^p[w]_{A_p}.$$
In particular, 
$$\sup_{R\in\mathcal R}\frac{1}{w(R)}\int_R|f(x)-f_R|w(x)dx\less 2^p[w]_{A_p}\|f\|_{bmo}$$
for any $f\in bmo$, $1<p<\infty$, and $w\in A_p(\R\times\R)$, where the suppressed constant does not depend on $f$, $p$, or $w$.
}
\end{remark}

\subsection{Discrete Triebel-Lizorkin spaces}

We consider now the invariance of certain discrete Triebel-Lizorkin spaces defined by Frazier and Jawerth \cite{FJ2}. Let $\mathcal Q_d \subset \mathcal Q$ be the collection of standard dyadic cubes in $\rn$ and let $w\in A_\infty^d=A_\infty(\mathcal Q_d,dx)$ be dyadic  version of $A_\infty$ defined in terms of $\mathcal Q_d$.

\begin{theorem}\label{dtlbmo}
Let $\alpha\in\R$, $0<q <\infty$, $0<p<\infty$.  Then for any sequence $\{s_Q\}_{Q\in\mathcal Q_d}$ indexed by dyadic cubes $\mathcal Q_d$, we have
\begin{align*}
\|\{s_Q\}\|_{\dot f_\infty^{\alpha,q}}&:=\sup_{P\in\mathcal Q_d}\(\frac{1}{|P|}\int_P\sum_{Q\in\mathcal Q_d:Q\subset P}(|Q|^{-1/2-\alpha/n}|s_Q|\chi_Q(x))^qdx\)^\frac{1}{q}\\
&\hspace{-1.5cm}\approx\sup_{P\in\mathcal Q_d}\[\frac{1}{w(P)}\int_P\(\sum_{Q\in\mathcal Q_d:Q\subset P}(|Q|^{-1/2-\alpha/n}|s_Q|\chi_Q(x))^q\)^\frac{p}{q}w(x)dx\]^\frac{1}{p},
\end{align*}
where $w(P)=\int_Pw(x)dx$.
\end{theorem}

\begin{proof}
Let $\mathfrak F$ be the collection of sequences indexed by the collection of dyadic cubes $\mathcal Q_d$ in $\R^n$.  Define
\begin{align*}
\Lambda(\{s_Q\},P)(x)=\sum_{Q\in\mathcal Q_d:Q\subset P}(|Q|^{-1/2-\alpha/n}|s_Q|\chi_Q(x))^q
\end{align*}
for $P\in\mathcal Q_d$, $\{s_Q\}\in \mathfrak F$, and $x\in\R^n$.  By \cite[Corollary 5.7]{FJ2}, it follows that 
\begin{align*} 
\|\{s_Q\}\|_{X^p}^p&=\sup_{P\in\mathcal Q_d}\frac{1}{|P|}\int_P\[\sum_{Q\in\mathcal Q_d:Q\subset P}(|Q|^{-1/2-\alpha/n}|s_Q|\chi_Q(x))^q\]^pdx\\
&\approx\sup_{P\in\mathcal Q_d}\[\frac{1}{|P|}\int_P\sum_{Q\in\mathcal Q_d:Q\subset P}(|Q|^{-1/2-\alpha/n}|s_Q|\chi_Q(x))^qdx\]^p\\
&=\|\{s_Q\}\|_X^p. 
\end{align*}
Therefore $X(\mathcal Q_d,\Lambda)$ satisfies the John-Nirenberg $p$-power inequality for the interval  $(0,\infty)$. 
Since the characterization $\|\{s_Q\}\|_{X^p}\approx\|\{s_Q\}\|_X$ is valid for all sequences $\{s_Q\}\in\mathfrak F$, meaning in particular that if $\|\{s_Q\}\|_{X^p}$ is finite for some $0<p<\infty$ then it is finite for all $0<p<\infty$, it follows that $X^p(\mathcal Q_d,\Lambda)=X(\mathcal Q_d,\Lambda)$ for all $0<p<1$.   The proof is completed by applying Theorem \ref{necessity}.   \end{proof}

\subsection{Triebel-Lizorkin spaces}

Fix a function $\varphi\in\S$  (the usual Schwartz space of smooth rapidly decreasing  functions) such that $\widehat\varphi(\xi)$ is supported in $1/2<|\xi|<2$ and $\widehat\varphi(\xi)\geq c_0>0$ for $3/5<|\xi|<5/3$.  Also define $\varphi_k(x)=2^{kn}\varphi(2^kx)$.  For $\alpha\in\R$, $0<q<\infty$, and $w\in A_\infty$, define the homogeneous $p=\infty$ type Trieble-Lizorkin space $\dot F_{\infty,w}^{\alpha,q}$ to be the collection of all $f\in\S'/\mathcal P$  (tempered distribution modulo polynomials) such that
\begin{align*}
\|f\|_{\dot F_{\infty,w}^{\alpha,q}}=\sup_{Q\in\mathcal Q}\(\frac{1}{w(Q)}\int_Q\sum_{k\in\Z:2^{-k}\leq\ell(Q)}(2^{\alpha k}|\varphi_k*f(x)|)^qw(x)dx\)^\frac{1}{q}<\infty,
\end{align*}
where $w(Q)=\int_Qw(x)dx$.  When $w\equiv 1$, we simply write $\dot F_\infty^{\alpha,q}=\dot F_{\infty,1}^{\alpha,q}$.

The following lemma is implicit in the work in \cite{FJ2} as it will be seen from its proof.

\begin{lemma}\label{TLJN}  
Let $\alpha\in\R$ and $0<p,q<\infty$.  Then for $f\in\S'/\mathcal P$, we have
\begin{align*}
&\sup_{Q\in\mathcal Q}\frac{1}{|Q|}\int_Q\(\sum_{k\in\Z:2^{-k}\leq\ell(Q)}(2^{\alpha k}|\varphi_k*f(x)|)^q\)^pdx\less\|f\|_{\dot F_\infty^{\alpha,q}}^{pq}.
\end{align*}
\end{lemma}

\begin{proof}
Define for $f\in\S'/\mathcal P$
\begin{align*}
\sup(f)_Q=|Q|^{1/2}\sup_{y\in Q}|\varphi_k*f(y)|
\end{align*}
for $Q\in\mathcal Q_d$, where $\ell(Q)=2^{-k}$.  Let $\alpha\in\R$ and $0<p,q<\infty$.  Then for any $f\in\S'/\mathcal P$, we have
\begin{align*}
&\sup_{P\in\mathcal Q}\frac{1}{|P|}\int_P\(\sum_{k\in\Z:2^{-k}\leq\ell(P)}(2^{\alpha k}|\varphi_k*f(x)|)^q\)^pdx\\
&\hspace{.1cm}=\sup_{P\in\mathcal Q}\frac{1}{|P|}\int_P\(\sum_{k\in\Z:2^{-k}\leq\ell(P)} \,\,\, \sum_{\substack{Q\in\mathcal Q_d:Q\subset P,\\\,\ell(Q)=2^{-k}}}
(2^{\alpha k}|\varphi_k*f(x)|\chi_Q(x))^q\)^p dx\\
&\hspace{.1cm}\leq\sup_{P\in\mathcal Q}\frac{1}{|P|}\int_P\(\sum_{Q\in\mathcal Q_d:Q\subset P}(|Q|^{-1/2-\alpha/n}\sup(f)_Q\chi_Q(x))^q\)^pdx\\
&\hspace{.1cm}\less\sup_{P\in\mathcal Q}\(\frac{1}{|P|}\int_P\sum_{Q\in\mathcal Q_d:Q\subset P}(|Q|^{-1/2-\alpha/n}\sup(f)_Q\chi_Q(x))^qdx\)^p\\
&\hspace{.1cm}=\|\sup(f)\|_{\dot f_\infty^{\alpha,q}}^{pq}.
\end{align*}
The last inequality here holds by \cite[Corollary 5.7]{FJ2} applied with $\{s_Q\}=\{\sup(f)_Q\}$.  Also by \cite[Theorem 5.2 and equation (5.6)]{FJ2}, it  follows that $\|\sup(f)\|_{\dot f_\infty^{\alpha,q}}\approx\|f\|_{\dot F_\infty^{\alpha,q}}$, which completes the proof.  
\end{proof}

\begin{theorem}\label{tlbmo}
Let $\alpha\in\R$, $0<q,p<\infty$, and $w\in A_\infty$.  Then for any $f\in \dot F_\infty^{\alpha,q}$, we have
\begin{align*}
\|f\|_{\dot F_{\infty}^{\alpha,q}} &  \approx\sup_{Q\in\mathcal Q}\(\frac{1}{w(Q)}\int_Q\(\sum_{k\in\Z:2^{-k}\leq\ell(Q)}\!\!\!(2^{\alpha k}|\varphi_k*f(x)|)^q\)^\frac{p}{q}\!\!w(x)dx\)^\frac{1}{p}.
\end{align*}
\end{theorem}

\begin{proof}
Let $\alpha\in\R$, $0<q,p<\infty$, $w\in A_\infty$, $\mathfrak F=\dot F_\infty^{\alpha,q}$ and 
\begin{align*}
\Lambda(f,Q)(x)=\sum_{k\in\Z:2^{-k}\leq\ell(Q)}(2^{\alpha k}|\varphi_k*f(x)|)^q.
\end{align*}
It follows that $X(\mathcal Q,\Lambda)=\dot F_\infty^{\alpha,q}$ and $\|\cdot\|_X=\|\cdot\|_{\dot F_\infty^{\alpha,q}}^q$.  By Lemma \ref{TLJN}, it follows that
\begin{align*}
&\sup_{Q\in\mathcal Q}\frac{1}{|Q|}\int_Q\(\sum_{k\in\Z:2^{-k}\leq\ell(Q)}(2^{\alpha k}|\varphi_k*f(x)|)^q\)^pdx\\
&\hspace{2cm}\less\sup_{Q\in\mathcal Q}\(\frac{1}{|Q|}\int_Q\sum_{k\in\Z:2^{-k}\leq\ell(Q)}(2^{\alpha k}|\varphi_k*f(x)|)^qdx\)^p
\end{align*}
for all $0< p<\infty$.  By Remark \ref{r:Xp=X}, it follows that $X(\mathcal Q,\Lambda)=X^p(\mathcal Q,\Lambda)$ for all $0<p<1$. 
 Then by Theorem \ref{necessity}, it follows that 
\begin{align*}
\|f\|_{\dot F_\infty^{\alpha,q}}&=\|f\|_X^{1/q}\approx\|f\|_{X_w^{p/q}}^{1/q}\\
&=\sup_{Q\in\mathcal Q}\[\frac{1}{w(Q)}\int_Q \(\sum_{k\in\Z:2^{-k}\leq\ell(Q)}(2^{\alpha k}|\varphi_k*f(x)|)^q\)^{p/q}w(x)dx\]^{1/p}
\end{align*}
for any $f\in\dot F_\infty^{\alpha,q}$, $0<p<\infty$, $\alpha\in\R$, $0<q<\infty$, and $w\in A_\infty$.
\end{proof}

\end{document}